\documentclass{article}

\usepackage{amsmath,amsthm,amsfonts}

\usepackage{fullpage}

\input{xy}
\xyoption{all}
\xyoption{matrix}

\def\cl#1{{\langle #1\rangle}}
\def\TO{\Longrightarrow}
\def\abel{\hbox{\footnotesize abelian}}
\def\nonabel{\hbox{\footnotesize non abel}}
\def\w{\wedge}

\def\om{\omega}
\def\ra{\overline}
\def\Aut{\mathrm{Aut}}
\def\Ker{\mathrm{Ker\ }}
\def\GL{\mathrm{GL}}
\def\SU{\mathrm{SU}}
\def\SL{\mathrm{SL}}
\def\tr{\mathrm{tr}}
\def\dimm{\hbox{{\footnotesize dim}}}

\def\g{\mathfrak{g}}
\def\r{\mathfrak{r}}
\def\h{\mathfrak{h}}

\def\su{\mathfrak{su}}
\def\sl{\mathfrak{sl}}
\def\aff{\mathfrak{aff}}
\def\D{\mathcal{D}}
\def\Z{\mathcal{Z}}
\def\R{\mathbb{R}}
\def\K{\mathbb{K}}
\def\C{\mathbb{C}}
\def\ie{{\em i.e. }}

\def\ad{\mathrm{ad}}
\theoremstyle{plain} 
\newtheorem{theorem}{Theorem}[section] 
\newtheorem{lemma}[theorem]{Lemma} 
\newtheorem{proposition}[theorem]{Proposition} 
\newtheorem{corollary}[theorem]{Corollary} 
\newtheorem{remark}[theorem]{Remark}
\date{}

\begin{document}
\title{Three dimensional real  Lie bialgebras}

\author{Marco A. Farinati
\thanks{mfarinat@dm.uba.ar. Member of CONICET. Partially supported by UBACyT X051 and PICT
2006-00836. } 
and A. Patricia Jancsa
\thanks{pjancsa@dm.uba.ar. Partially supported by UBACyT X051 and PICT 2006-00836. } 
}
 
\maketitle
\begin{abstract}
We classify all real three dimensional Lie bialgebras. In each case, their automorphism group
as Lie bialgebras is also given.
\end{abstract}

\section*{Introduction}

Our goal is to classify the {\em real} three dimensional Lie bialgebras. 
Recall that a Lie bialgebra over a field $\K$ is a triple $(\g,[-,-],\delta)$
where $(\g,[-,-])$ is a Lie algebra over $\K$
 and $\delta:\g\to\Lambda^2\g$
is such that 
\begin{itemize}\itemsep=-0.5ex
 \item  the induced map $\delta^*:\Lambda^2\g^*\subseteq(\Lambda^2\g)^*\to \g^*$
is a Lie algebra structure on $\g^*$, 
 \item $\delta:\g\to\Lambda^2\g$ is a 1-cocycle in the Chevalley-Eilenberg
complex of the Lie algebra $(\g,[-,-])$ with coefficients in $\Lambda^2\g$.
\end{itemize}
The Jacobi condition for $\delta^*$ is called co-Jacobi condition.
%It is  usually  computed as the alternate sum of $(\delta\otimes \mathrm{id})\delta =0$.
A Lie bialgebra is said factorizable,
if there exist $r\in\g\otimes \g$ such that $\delta (x)=\ad _x(r)$ $\forall x\in\g$,
$r$ satisfies the classical Yang-Baxter equation
and the symmetric component of $r$ induces a nondegenerate inner product on $\g ^*$. According to \cite{AJ}, a real Lie bialgebra is said almost factorizable if the complexification is factorizable.

In \cite{G}, the author gives a classification of three dimensional Lie bialgebras,
but for example, in the $\sl(2,\R)$ case, we find differences between his result and our.
Namely, we find for $\sl(2,\R)$, three families of isomorphism classes of Lie bialgebras, apart from the
 co-abelian
(see section \ref{seccionsl2} for the details).
Our first case is a 1-parameter family: the cobracket is a positive multiple 
of $\ad(r)$ with $r=y\w x$ in the standard basis $\{h,x,y\}$. The multiple can always
be chosen positive because we prove that $(\g,\delta)$ and $(\g,-\delta)$ give isomorphic
Lie bialgebras in this case, while in \cite{G} they appear as non isomorphic.
 Our second 1-parameter family coincides with the one in \cite{G}.
On the other hand, for the third type,
$r_{\pm}=\pm \frac12 h\w x$ give two non-isomorphic (triangular) Lie bialgebras
on $\sl(2,\R)$, unlike the situation in \cite{G} where only one possibility of sign is 
found.

In \cite{RA-H-R}  all the $r$-matrices for real 3-dimensional Lie algebras are computed;
but for 3-dimensional solvable Lie algebras $H^1(\g,\Lambda^2\g)$ is not trivial.
 Besides, in our work, we distinguish the isomorphism classes of Lie bialgebras.
Although we do not find explicitly the $r$-matrices in the cases of coboundary Lie bialgebras,
it is not hard to compute them. We compute 
 $(\Lambda ^2\g) ^\g$ and all the 1-cocycles, which imply the computation of
$\dim H^1(\g,\Lambda^2 \g)$ (see table below), since the space of coboundaries is isomorphic to 
$\Lambda ^2 \g/(\Lambda ^2 \g)^\g$.

\begin{center}
{\em \small Dimension of $H^1(\g,\Lambda^2\g)$ for real 3-dimensional Lie algebras}
\end{center}
\[
\begin{array}{|c|c|c|c|c|c|c|c|c|c|}
\hline
\g&\h_3&\r_3&\r_{3,\lambda}&\r_{3,\lambda}&\r_{3,\lambda}
&\r_{3,\lambda}'&\r_{3,\lambda}'&\su(2)& \sl(2,\R)\\
&&&{}^{\lambda\neq \pm1}&{}^{\lambda=-1}&{}^{\lambda=1}
&{}^{\lambda\neq 0}&{}^{\lambda=0}&&\\
\hline
\dimm(\Lambda^2\g)^\g &2&0&0&1&0&0&1&0&0\\
\dimm( \hbox{1-{\footnotesize coboundaries}})&1&3&3&2&3&3&2&3&3\\
\dimm( \hbox{1-{\footnotesize cocycles}})&6&4&4&4&6&4&4&3&3\\
\dimm (H^1(\g,\Lambda^2\g))&5&1&1&2&3&1&2&0&0\\
\hline
\end{array}
\]

Our method is direct: we fix a Lie algebra structure,  find all possible
1-cocycles,  solve the co-Jacobi condition and let the Lie algebra
automorphisms group  act on the set of solutions; in this way we find 
simultaneously the isoclasses of Lie bialgebras and its automorphism group 
as Lie bialgebras.
Our main result is the complete classification of the real 3-dimensional Lie bialgebras, which is given case by case in each section; the Lie bialgebras automorphisms groups 
are given as well.

Both authors thank M.L. Barberis for helpful comments 
 on the automorphism groups   and very  fruitful discussions.

%{\footnotesize \tableofcontents}

\section{General results}

{\bf The center}.
Given a Lie bialgebra $(\g,[-,-],\delta)$, if
one fixes the Lie algebra structure and varies $\delta$, the 1-cocycle condition
can be viewed as a set of linear equations in the
matrix coefficients of the cobracket. %, and from this point of view it is easy to solve. 
Anyway, in some cases, the
following property simplifies computations:

\begin{proposition}\label{center}
 Let $\g$ be a Lie algebra and $\delta:\g\to\Lambda^2\g$ a 1-cocycle, 
then $\delta(\Z \g)\subseteq (\Lambda^2\g)^\g$.
\end{proposition}
\begin{proof}
 Let $z\in\Z\g$ and $x\in\g$ arbitrary, the 1-cocycle condition reads
\[
 \delta[x,z]=[\delta x,z]+[x,\delta z]
\]
But $z\in\Z\g$ implies $[x,z]=0$ for any $x$, and also $[z,\Lambda^2\g]=0$, so we conclude
$[x,\delta z]=0$, $ \forall x\in \g$,
namely, $\delta z\in(\Lambda^2\g)^\g$.
\end{proof}
The above proposition will be useful when
%, for example, $\Z\g\neq 0$ and $(\Lambda^2\g)^\g=0$, or when  
$\Z\g$ is ``big'' and $(\Lambda^2\g)^\g$ ``small''. So, 
%working with each of for the 3-dimensional Lie algebras, in each case 
it will be useful
to start computing the center and the invariant part or $\Lambda^2\g$.

\

\noindent {\bf The derived ideal $ [\g, \g ]$\label{sgg}}.
Recall that a coideal in a Lie bialgebra $\g$ is a subspace $V\subseteq \g$ such that  $\delta V\subseteq V\wedge \g$.
Such a subspace occurs as kernel of a Lie coalgebra map.
The 1-cocycle condition for $\delta$ implies the following:

\begin{proposition} 
\label{gg}
Let $(\g,\delta)$ be a Lie bialgebra, then $[\g,\g]$ is a coideal. In particular, 
the quotient $\g / [\g, \g ]$ admits a unique Lie bialgebra structure such that the canonical projection is a Lie bialgebra map. 
Moreover, if  $(\g,\delta _1)\cong (\g,\delta _2)$ as Lie bialgebras, then 
$(\g /[\g,\g],\ra{\delta }_1)\cong (\g /[\g,\g],\ra{\delta} _2)$.
\end{proposition}
Notice that the Lie algebra structure on $\g/[\g,\g]$ is trivial, so a Lie bialgebra structure on
$\g/[\g,\g]$ is equivalent to an usual Lie algebra structure on
$(\g/[\g,\g])^*$.

%\begin{corollary} If  $(\g,\delta _1)\cong (\g,\delta _2)$ as Lie bialgebras, then 
%$(\g /[\g,\g],\ra{\delta }_1)\cong (\g /[\g,\g],\ra{\delta} _2)$.
%\end{corollary}

\begin{lemma}\label{lgg} Let $\g$ be a Lie algebra and $\psi :\g\to\g$ a Lie algebra automorphism, then 
$\psi$ induces Lie algebra morphisms $\psi |_{[\g,\g]} :[\g,\g]\to[\g,\g]$ and
 $\ra{\psi }:\g /[\g,\g]\to\g /[\g,\g]$. The applications $\Aut(\g)\to\Aut([\g,\g])$ and
$\Aut(\g)\to\Aut(\g /[\g,\g])$ defined by $\psi \mapsto\psi |_{[\g,\g]}$ and
$\psi \mapsto\ra{\psi}$ are group homomorphisms.
\end{lemma}

\begin{remark}
Proposition \ref{gg} says that by trivializing the bracket one gets a quotient Lie bialgebra.
 The dual statement of Proposition \ref{gg} is about a subobject of $\g$ instead of a quotient:
$\Ker\delta$ is a Lie subalgebra (due to the 1-cocycle condition) and it is 
obviously maximal with respect to the property of having trivial cobracket. If
$\g_1\cong \g_2$ are two isomorphic Lie bialgebras, then
$\Ker\delta_1 \cong \Ker\delta_2$ as Lie algebras and also as  bialgebras with trivial cobracket.
\end{remark}

\noindent 
{\bf The characteristic bi-derivation}.
Let $(\g,[-,-],\delta)$ be a Lie bialgebra. 
%In this situation one can consider the composition
The characteristic endomorphism $\D:\g\to \g$ given by % the composition
\[
\xymatrix{
\g\ar[r]^{\delta}\ar@<-0.5ex>@/_1pc/[rr]_{\D:=[-,-]\circ\delta}&\Lambda^2\g\ar[r]^{[-,-]}&\g
}
\]
%and get a characteristic endomorphism $\D:\g\to \g$, in the sense that
is clearly preserved by Lie bialgebra isomorphisms. Namely, if $\phi:\g\to\g'$ is a Lie bialgebra isomorphism
and $\D_{\g}$ and $\D_{\g'}$ denote the endomorphism associated to $\g$ and $\g'$ respectively,
then $\D_{\g'}=\phi\D_{\g}\phi^{-1}$. As a consecuence, any function in $\D$,
which is invariant under conjugation, provides an invariant of the isomorphism class of the Lie bialgebra.
For example, $\det(\D)$ and  $\tr(\D)$ are (real)  numerical invariants. The characteristic polynomial
of $\D$ and its Jordan form are also invariants.
Lie bialgebras  $\g$ such that $\D_\g=0$ are called {\bf involutive}, but, in many cases $\D$ is far from being zero.
The following proposition is standard. %maybe folklore:
\begin{proposition}
 Let $\g$ be a Lie bialgebra and $\D=[-,-]\circ\delta$. Then $\D$ is a derivation
with respect to the bracket and a coderivation with respect to the cobracket.
\end{proposition}
\begin{proof}
 We will prove that $ \D$ is a derivation, the second claim follows by dualization.
Let us write the 1-cocycle condition using Sweedler-type notation:
$
\delta[x,y]=[\delta x,y]+[x,\delta y] =[x_1\wedge x_2,y]+[x,y_1\wedge y_2]
=[x_1,y]\wedge x_2 +x_1\wedge [x_2,y]+[x,y_1]\wedge y_2+y_1\wedge [x, y_2]$.
Then it follows that
$%\[
 \D[x,y]
=[[x_1,y],x_2] +[x_1,[x_2,y]]+[[x,y_1], y_2]+[y_1,[x, y_2]]
$. %\]
Using Jacobi identity and the definition of $\D$, we get
\[ \D[x,y]
=[[x_1,x_2],y] +[x,[y_1, y_2]]
=[\D x,y]+[x,\D y]
\]
\end{proof}

\section{Two dimensional Lie bialgebras\label{dim2}}

In a similar way that one proves that there are only two isoclasses of Lie algebras of dimension 2, an 
easy manipulation of basis
shows that the following list exhausts the isoclases of two dimensional Lie bialgebras. The
structure is given in a basis $\{h,x\}$ of $\g$.
\begin{center}{\em {\small{2-dimensional Lie bialgebras isomorphism classes}}}\end{center}
\[
\begin{array}{|c|c|c|c|c|c|c|}
\hline
\g&\g^*&[-,-]&\delta&\hbox{Invariants}&\hbox{Name}&\tr(\D)\\
\hline
\abel &\abel&0&0&&&0\\
\hline
\abel      &\nonabel&0       &\delta h=x\wedge h; \delta x=0&&&0\\
\hline
\aff(\R)& \abel    &[h,x]=x& 0 & &&0\\
\hline
\aff(\R)& \nonabel&[h,x]=x&\delta h=h\wedge x;
\delta x=0
&\Ker\delta=[\g,\g];&\aff_{2,0}&\\
&&&&\delta=\partial r, r=h\wedge x&&0\\
\hline
\aff(\R)& \nonabel&[h,x]=x&\delta h=0;
\mu\neq 0&\Ker\delta\neq [\g,\g]&\aff_{2,\mu}&\mu\\
&&& \delta x=\mu h\wedge x;
&&&\\
\hline
\end{array}
\]
where $\aff(\K)$ is the non-abelian 2-dimensional  Lie algebra over $\K$.

%\noindent 
The first four lines are clearly non isomorphic among them, 
and non isomorphic to any of the last line.
Finally, thanks to the invariant given by the trace of
the characteristic derivation, one sees that they are not isomorphic
to each other for different $\mu$.
The same table is valid for any field $\K$, replacing $\aff(\R)$
by $\aff(\K)$.

\begin{remark}A similar table appears in \cite{K-S}, but without the parameter $\mu$, which can not be
eliminated, because $\tr(\D)$ is an invariant  of the Lie bialgebra.
\end{remark}
Simply by inspection,  notice the following:
\begin{proposition}If $(\g,[-,-],\delta)$ is a Lie bialgebra with $\dim\g=2$ and 
$\D=[-,-]\circ\delta$, then, whithin the non abelian and non co-abelian cases, 
$\tr(\D)$ is a total invariant. With notation as in the table, $\aff_{2,\mu}\cong
\aff_{2,\mu'}$ if and only if $\mu=\mu'$.
\end{proposition}

\begin{corollary}
Let $a,b,c,d\in\R$ such that $(a,b)\neq (0,0)$ and $(c,d)\neq (0,0)$; consider
$\g_{abcd}$ the Lie
bialgebra  given by
$%\[
[h,x]=ah+bx,\
%\]\[
\delta h=ch\w x,\ \delta x=dh\w x$, %\]
then $\g_{abcd}\cong \aff_{2,\mu}(\R)$ with $\mu=ac+bd$.
\end{corollary}

\begin{proof}Since $(a,b)\neq (0,0)$ and $(c,d)\neq (0,0)$ we are not in the abelian or
co-abelian case. It suffices to compute the trace of $\D$. 
The computations
$\D(h)=c[h,x]$ and  $\D(x)=d[h,x]$ give $\tr(\D)=ac+db$.
\end{proof}

%\subsection
{\bf Automorphisms groups in the non abelian and non co-abelian cases}.
Consider the ordered basis $\{h,x\}$ of $\g$,
then
the Lie bialgebra automorphisms groups in the non abelian  and non co-abelian cases are as follows:
%\begin{itemize} \item 

$\bullet$ Case $\g=\aff(\R)$ with $[h,x]=x$ and $\delta h=h\wedge x; \ \delta x=0 $; 
\[
\Aut(\g)=\left\{  
\left(\begin{array}{cc}
1 &0 \\
b &1\\
\end{array}
\right):\ b\in\R
\right \}
\]
\noindent In particular, any of these maps is the exponential of a multiple of \\ %$\D$ where 
$%\[
\D=[-,-]\circ\delta=
\left(\begin{array}{cc}
0 &0 \\
1 &0\\
\end{array}
\right)
$. %\]
%\item 

$\bullet$ Case $\g=\aff(\R)$ with $[h,x]=x$ and $\delta_\mu h=0;$ $ \delta_\mu x
=\mu h\wedge x $:
\[
\Aut(\g)=\left\{  
\left(\begin{array}{cc}
1 &0 \\
0 &d\\
\end{array}
\right):\ 0\neq d \in\R 
\right \}
\]
\noindent Any of these maps with $d>0$ is the exponential of a multiple of  \\
%$\D_{\mu}$ where 
$%\[
\D_{\mu}=[-,-]\circ\delta_{\mu}=
\left(\begin{array}{cc}
0 &0 \\
0 &\mu\\
\end{array}
\right)
$. %\]
The fact that the exponential of (a multiple of) the endomorphism $\D$ gives an automorphism
of the Lie bialgebra is not surprising, since we already knew that
$\D$ is a derivation and a coderivation.
%\end{itemize}

\section{Three dimensional real Lie algebras}

\begin{theorem}\cite{GOV}
The following list exhausts the 3-dimensional solvable real Lie algebras:
\[
\begin{array}{rl}
\R^3:&\hbox{ the  three dimensional abelian;}\\
\h_3:& [e_1,e_2]=e_3, \hbox{ the three dimensional Heisenberg;}\\
\r_3:&[e_1,e_2]=e_2,\ [e_1,e_3]=e_2+e_3;\\
\r_{3,\lambda}:&[e_1,e_2]=e_2,\ [e_1,e_3]=\lambda e_3,\ |\lambda|\leq 1;\\
\r'_{3,\lambda}:&[e_1,e_2]=\lambda e_2-e_3,\ [e_1,e_3]=e_2+\lambda e_3,\ \lambda\geq 0.\\
\end{array}\]
Denote $u=\frac{ih}{2}$, $v=\frac{x-y}{2}$, $w=\frac{i(x+y)}{2}$;
the semisimple 3-dimensional real Lie algebras are
\[
\sl(2,\R): [h,x]=2x,\ [h,y]=-2y,\ [x,y]=h
\]
\[
\su(2): [u,v]=w,\ [v,w]=u,\ [w,u]=v.
\]
%$u=ih/2$, $v=(x-y)/2$, $w=i(x+y)/2$.
\end{theorem}

%\subsection
{\bf Three dimensional real Lie bialgebras: general strategy}.
In order to classify all real three dimensional Lie bialgebras we will procede as follows:
\begin{enumerate}
\item Given a Lie algebra $\g$, we find the general 1-cocycle $\delta:\g\to\Lambda^2\g$.
\item  Determine when $\delta$ satisfies the co-Jacobi identity.
\item Study the action of $\Aut(\g,[-,-])$ on the set of cobrackets $\delta$.
\item Find a set of representatives, hence, the list of isomorphism classes of
Lie bialgebras with underlying Lie algebra $\g$.
\end{enumerate}

%The easiest case is the abelian  Lie algebra $\R^3$. 
To give a Lie bialgebra structure on the abelian Lie algebra $\R^3$ is the same
as giving a Lie algebra structure on $(\R^3)^*$, so the list of
all three dimensional Lie bialgebras with underlying Lie algebra $\R^3$
is in obvious bijection with the list of three dimensional Lie algebras.

Next, we proceed with the other cases: first $\h_3$, the
only 3-dimensional nilpotent and non abelian Lie algebra, secondly the solvable non nilpotent
$\r_3$, $\r_{3,\lambda}$ and $\r'_{3,\lambda}$, and finally the simple 
$\su(2)$ and $\sl(2,\R)$. %Before doing this, let us fix some general notation.
% for $\delta$ that is going to be used in all cases:

\subsection{The general co-Jacobi condition\label{gcJc}}

%In general, 
If $\g$ is any three dimensional Lie algebra, we will write the structure in terms of basis
$\{x,\ y, \ h\}$ of $\g$ and
$\{x\wedge y,\ y\wedge h, \ h\wedge x\}$ for $\Lambda ^2 (\g)$.
Write, with $a_i$, $b_i$, $c_i\in\R$, $i=1,2,3$,
$\delta x=a_1x\wedge y+a_2y\wedge h+a_3h\wedge x$;\\
$\delta y=b_1x\wedge y+b_2y\wedge h+b_3h\wedge x$;
$\delta h=c_1x\wedge y+c_2y\wedge h+c_3h\wedge x$.

\noindent 
For a linear map $\delta:\g\to\Lambda^2\g$, %then 
the co-Jacobi condition is equivalent to the equations
\begin{eqnarray*}
-a_1b_2+a_2 (b_1-c_3)+a_3 c_2=0,
\\
b_1 a_3-b_2 c_3+b_3 (-a_1+c_2)=0,
\\
c_1(a_3-b_2)  +c_2 b_1-c_3a_1=0.
\end{eqnarray*}

%[[[[[[[[[[[[[[[[[[[[[[[[[[[[[[[[[[[[[[[[[[[[[[[[[[[[[[[[[[[[[[[[[[[[[
\section{Lie bialgebra structures on $\h_{3}$}
%[[[[[[[[[[[[[[[[[[[[[[[[[[[[[[[[[[[[[[[[[[[[[[[[[[[[[[[[[[[[[[[[[[[[[

Recall that the Lie algebra $\h _{3 }$ has a basis $\{x,\ y, \ h\}$, with the relations
$[h,x]=0$, $[h,y]=0$,  $[x,y]=h$. We list general properties of $\h_3$:
\begin{itemize}
\item 
$[\h_3,\h_3]=\R h$; $[\h_3,[\h_3,\h_3]]=0$.
\item $\Z (\h_3)=\R h$, $(\Lambda^2\h_3)^{\h_3}=\R y\wedge h\oplus\R h\wedge x$.
\item The automorphisms group of  $\h _{3 }$
%this Lie algebra in the oredered basis $\{x,\ y, \ h\}$ 
is the following 
subgroup of $\GL (3,\R)$:
\[
 \Aut(\h _3)=\left\{\phi _{\mu,\rho,\sigma,\nu,a,b}  =
\left(\begin{array}{ccc}
\mu &\rho & 0\\
\sigma &\nu &0\\
a &b & \lambda
\end{array}
\right):\ \mu \nu -\rho\sigma =\lambda\neq 0 
\right \}
\]
\end{itemize}

{\bf The 1-cocycle condition}.
Consider the basis $\{x\wedge y,\ y\wedge h, \ h\wedge x\}$ of $\Lambda ^2 (\h_3)$ and
write $\delta$ as in  \ref{gcJc}. Proposition
\ref{center} implies $\delta h=c_2y\wedge h+c_3h\wedge x$, namely 
$c_1=0$. The 1-cocycle condition for $[h,x]$ and $[h,y]$ is the content of the
proof of this proposition, so it gives no further information  in this case.
Besides, the 1-cocycle for $[x,y]=h$ reads $\delta h=\delta[x,y]=[\delta x,y]+[x,\delta y]$,  then 
%\[
\begin{eqnarray*}
\delta h&=&c_2y\wedge h+c_3h\wedge x\\
&=&
[a_1x\wedge y+a_2y\wedge h+a_3h\wedge x,y]+
[x,b_1x\wedge y+b_2y\wedge h+b_3h\wedge x]\\
&=&a_1h\wedge y+ b_1x\wedge h
%\hskip 13.2cm
\end{eqnarray*}
%\]
so  $c_2=-a_1$ and $c_3=-b_1$.
 Hence, then general  1-cocycle  is 
\[\begin{array}{rcl}
\delta (x)&=&a_1x\wedge y+a_2 y\wedge h+ a_3h\wedge x;\\
\delta (y)&=&b_1x\wedge y +b_2y\wedge h+b_3h\wedge x;\\
\delta (h)&=&%0x\wedge y
\quad\quad\quad - a_1y\wedge h-b_1 h\wedge x
\end{array}
\]
In matrix notation, 
%\[
$\delta =
\left(\begin{array}{ccc}
a_1&b_1&0\\
a_2& b_2 &-a_1\\
a_3& b_3 & -b_1
\end{array}\right).
$%\]

%\subsection{The co-Jacobi condition for a 1-cocycle in $\h_3$}

The  co-Jacobi condition (see \ref{gcJc}) restricted to a 1-cocycle in $\h_3$ reduces to
\begin{eqnarray*}
2 a_2 b_1-a_1 (a_3+b_2)&=0,\\
b_1 a_3+ b_1 b_2-2 a_1 b_3&=0,
\end{eqnarray*}
%This set of equations is 
which are not so easy to solve, so we use a dimensional reduction procedure, thanks to the results of section \ref{sgg}.

\subsection{Consequence of the general result for $\h_3 / [\h_3, \h_3 ]$}

\begin{lemma} For $\g =\h_3$, the natural application of Lemma
\ref{lgg}\\
$\Aut(\g)\to\Aut(\g /[\g,\g])$, defined by 
$\psi \mapsto\ra{\psi}$, is a  split epimorphism.
\end{lemma}
\begin{proof}Consider the basis $\{x,y,h\}$, the splitting may be defined as
%\[
\begin{eqnarray*}
\Aut(\g /[\g,\g])&\to &\Aut(\g)\\
\phi =
\left(\begin{array}{cc}
\mu &\rho \\
\sigma &\nu 
\end{array}
\right)
&\mapsto &
\widehat{\phi }=
\left(\begin{array}{cc|c}
\mu &\rho & 0\\
\sigma &\nu &0\\
\hline
0 &0 & \mu\nu-\rho\sigma
\end{array}
\right)
\end{eqnarray*}
%Namely, if $\phi\ra x=\mu \ra x+\rho\ra y$ and
% $\phi\ra y=\sigma\ra x+\nu\ra y$ is an automorphism of the (abelian) Lie algebra $\h_3/\cl{h}$, then
%  $\wh\phi x=\mu  x+\rho y$, $\wh\phi y=\sigma x+\nu y$
%and $\wh\phi h=(\det\phi) h$ defines an automorphism of $\h_3$, and it can be easyly verified 
%that $\phi\mapsto\wh\phi$ is a group homomorphism..
\end{proof}

We know $\dim( \h_3 /[\h_3, \h_3 ])=2$ %; moreover, we know also
and  $(\h_3 /[\h_3, \h_3 ])$ is abelian.
According to the  section \ref{dim2}, %about 2-dimensional Lie bialgebras, 
there are only 
 two classes of isomorphisms of 2-dimensional Lie bialgebras with abelian bracket:
 the co-abelian one and the one with $\ra{\delta }(\ra x)=0$ and $\ra{\delta }(\ra y)=\ra{x}\wedge \ra{y}$.
Observe that, in virtue of the form of the Lie algebra automorphims, there is no lost of generality in 
assuming that the basis $\{\ra{x},\ \ra{y}\}$ is the one which allows us to write $\ra{\delta}$ in this form since 
any automorphism of $( \h_3 /[\h_3, \h_3 ])$ may be lifted to an automorphism of $\h_3$.
Explicitly, we may assume $a_1=0$ and there are two possibilities for $b_1$, namely, $b_1=0$ or $b_1=1$.
In matrix notation,
\[
\delta _{b_1=0}=
\left(\begin{array}{ccc}
0&0&0\\
a_2& b_2 &0\\
a_3& b_3 & 0
\end{array}\right)
\ \hbox{ and }\
\delta _{b_1=1}=
\left(\begin{array}{ccc}
0&1&0\\
a_2& b_2 &0\\
a_3& b_3 & -1
\end{array}\right).
\]
Returning to the  co-Jacobi condition with the assumption $a_1=0$, %we see that 
it is automatically satisfied in the case $b_1=0$, and
it reduces to $b_2+a_3=0=2a_2$ if $b_1=1$.

%\subsubsection{Case $b_1=1$}

\noindent {\em Case $b_1=1$}.
In this case,
$\delta=
\left(
\begin{array}{ccc}
0&1&0\\
0&-a_3&0\\
a_3&b_3&-1
\end{array}\right)
$.
After conjugation 
by $\phi _{\mu, \nu, \rho , \sigma, a, b}$, we get
\[
\delta'=
\frac{1}{\mu \nu - \rho \sigma}
\left(\begin{array}{ccc}
\sigma&\nu&0\\
\frac{ b_3\sigma^2}{\mu\nu - \rho\sigma}&
\frac{a   \nu - a_3\mu\nu  - b\sigma + b_3\nu\sigma + a_3\rho \sigma}
{\mu\nu - \rho \sigma}&-\sigma\\
\frac{-a\nu + a_3\mu\nu  + b\sigma + b_3\nu\sigma - a_3\rho\sigma}
{\mu\nu - \rho\sigma}&\frac{b_3\nu^2 }{\mu\nu - \rho\sigma}&
-\nu\\
\end{array}\right)
\]
If one wants to preserve the condition $a'_1=0$ then it must be $\sigma=0$, so %hence
$%\[
\delta'=
\left(\begin{array}{ccc}
0&\frac{1}{\mu}&0\\
0&\frac{a   - a_3\mu }{\mu^2\nu }&0\\
\frac{-a + a_3\mu  }{\mu^2\nu}&\frac{b_3 }{\mu^2 }
&-\frac{1}{\mu}\\
\end{array}\right)
$. %\]
The condition $b'_1=1$ forces $\mu=1$, so, with $\sigma=0$ and $\mu=1$,
$
\delta'=
\left(\begin{array}{ccc}
0&1&0\\
0&\frac{a - a_3 }{\nu }&0\\
\frac{-(a- a_3) }{\nu}&b_3&-1\\
\end{array}\right)
$.\\
Taking an automorphism with $a=a_3$ we get $a_3'=0$, namely
$\delta$ changes into
\[
\delta'=\delta_{b_3}=
\left(\begin{array}{ccc}
0&1&0\\
0&0&0\\
0&b_3&-1\\
\end{array}\right)
\]
An automorphim preserving also $a_3'=0$
must have $a=0$, and in this case $\delta'=\delta$.
%We conclude that 
Finally, the list of isoclasses of Lie bialgebras is
given by the cobrackets $\{\delta_{b_3}:b_3\in \R\}$ given above. %of the form
For each of these, the automorphism group of Lie bialgebras is
\[
G=
\left\{\phi_{\rho,\nu, b}=
\left(\begin{array}{ccc}
1&\rho&0\\
0&\nu&0\\
0&b&\nu\\
\end{array}\right):\nu\neq 0,b,\rho\in \R\right\}\]

\

\noindent {\em Case $b_1=0$}:
In this situation, co-Jacobi is automatically satisfied.
Let $\delta =
\left(
\begin{array}{ccc}
0&0&0\\
a_2&b_2&0\\
a_3&b_3&0
\end{array}
\right)
$ and let us compute $\delta ':=(\phi\wedge\phi)^{-1}\delta\phi$ with 
$\phi =\phi _{\mu, \nu, \rho , \sigma, a, b}$  then
\[
\delta' =
\frac{1}{(\mu\nu-\rho\sigma)^{2}}
\left(
\begin{array}{ccc}0 &0&0\\
a_2\mu ^2+\sigma \mu (a_3 +b_2)+b_3\sigma ^2&
b_2 \mu \nu+a_2 \mu \rho+b_3 \nu \sigma+a_3 \rho \sigma
 & 0\\
a_2 \mu \rho+b_2 \rho\sigma +b_3 \nu \sigma+a_3 \nu \mu
&  b_3 \nu^2+(a_3+b_2)\rho \nu+a_2 \rho^2 & 0\\
\end{array}
\right)
\]

\noindent Although the matrix $\left(
\begin{array}{cc}
a_2&b_2\\
a_3&b_3
\end{array}\right)
$ does not correspond to a symmetric bilinear form, it changes according to the following rule:
\[
\left(
\begin{array}{cc}
a_2'&b'_2\\
a'_3&b'_3
\end{array}\right)
=
\frac{1}{(\mu\nu-\rho\sigma)^{2}}
\left(
\begin{array}{cc}
\mu&\sigma\\
\rho&\nu
\end{array}\right)
\left(
\begin{array}{cc}
a_2&b_2\\
a_3&b_3
\end{array}\right)
\left(
\begin{array}{cc}
\mu&\rho\\
\sigma&\nu
\end{array}\right)
\]
Namely, it changes as a bilinear form, divided by the square of the determinant $\mu\nu  -\rho\sigma $.
We know that it can be diagonalized; hence, we may assume that, up
to isomorphism, $b_2=a_3=0$.
Also, with an automorphism with $\rho=\sigma=0$, the coefficients change according to the rule
$a'_2=a_2/\nu^2$, $a'_3=0$, $b'_2=0$, $b'_3=b_3/\mu^2$, so the full list of possibilities,
up to isomorphism, are
$a_2=0,\pm1$ and $b=0\pm 1$.
Using the automorphism with $\mu=\nu=0$, $\rho=\sigma=1$ we get $a_2'=b_2$ and $b_2'=a_3$,
so the Lie bialgebra with cobracket with $a_2=1=-b_3$ is isomorphic to the one with cobracket with $a_2=-1=-b_3$.

We also know that the signature is an invariant of bilinear forms and  it is an invariant also in this case,
because the difference between the action on bilinear forms and our case is the multiplication by the square
of the determinant, which is a positive number. Similar consideration holds for the rank.
 We conclude that the list of isomorphism classes is given by $\delta _{a_2, b_3}$ obtained by
choosing the parameters $(a_2,b_3)=(0,0),\ (1,1)$, $ (-1,-1)$, $ (1,-1)$, $ (1,0)$, $ (-1,0)$.
This completes the proof of the  following statement.

\begin{theorem}
For the Lie algebra $\h_3$ the exhaustive list
of  the isomorphism classes of Lie bialgebra structures is parametrized by the following set of cobrackets:
\[\delta _{b_3}=
\left(\begin{array}{ccc}
0&1&0\\
0&0&0\\
0&b_3&0\\
\end{array}\right):b_3\in\R
,\ \hbox{ and }
\delta _{a_2,b_3}=
\left(
\begin{array}{ccc}
0&0&0\\
a_2&0&0\\
0&b_3&0
\end{array}
\right)
\]
with
$(a_2,b_3)=(0,0),\ (1,1)$, $ (-1,-1)$, $ (1,-1)$, $ (1,0)$, $ (-1,0)$.

\end{theorem}

%%%%%%%%%%%%%%%%%%%%%%%%%%
\section{Lie bialgebra structures on $\r _{3}$}
%%%%%%%%%%%%%%%%%%%%%%%%%%

Recall  $\r_3$ is the real Lie algebra with basis $\{x,y,h\}$ and Lie brackets
given by \[
          [h,x]=x,\ [h,y]=x+y,\ [x,y]=0\]

\begin{lemma} 
$\Z (\r_3)=0$ and $(\Lambda ^2\r_3)^{\r_3}=0$.
\end{lemma}
\begin{proof}
$\ad_x(ax+by+ch)=-cx=0$ implies $c=0$
and $0=\ad_h(ax+by)=ax+b(x+y)$ implies $a=0=b$.
If $\om= ax\w y+bx\w h+cy\w h
\in(\Lambda ^2\r_3)^{\r_3}$ then
\[
\ad _h(\om)=2ax\w y+(b+c)x\w h+cy\w h=0
\]
implies $a=b=c=0$.
\end{proof}

{\bf 1-cocycles on $\r_3$.}
Let us write, as in \ref{gcJc},  with $a_i$, $b_i$, $c_i\in\R$,
$\delta (x)=a_1x\wedge y+a_2 y\wedge h+ a_3h\wedge x$,
$\delta (y)=b_1x\wedge y +b_2y\wedge h+b_3h\wedge x$,
$\delta (h)=c_1x\wedge y+c_2y\wedge h+c_3 h\wedge x$.
%\[\begin{array}{rcl}
%\delta (x)&=&a_1x\wedge y+a_2 y\wedge h+ a_3h\wedge x\\
%\delta (y)&=&b_1x\wedge y +b_2y\wedge h+b_3h\wedge x\\
%\delta (h)&=&c_1x\wedge y+c_2y\wedge h+c_3 h\wedge x
%\end{array}\]
The 1-cocycle  condition for $[x,y]=0$ and $\delta[x,y]=[\delta x,y]+[x,\delta y]$ give
\[\begin{array}{rcl}
0&=&[a_1x\wedge y+a_2 y\wedge h+ a_3h\wedge x,y]
+[x,b_1x\wedge y+b_2 y\wedge h+ b_3h\wedge x]
\\
&=&a_2 y\wedge x+ a_3y\wedge x-b_2 y\wedge x
\end{array}\]
so $a_2+a_3-b_2=0$.
Now, $[h,x]=x$ and $ \delta[h,x]=[\delta h,x]+[h,\delta x]$ imply 
\[\begin{array}{rcl}
&&a_1x\wedge y+a_2 y\wedge h+ a_3h\wedge x\\
&=&[c_1x\wedge y+c_2y\wedge h+c_3 h\wedge x,x]
+[h,a_1x\wedge y+a_2y\wedge h+a_3 h\wedge x]
\\
&=&c_2y\wedge x
+2a_1x\wedge y+a_2(x+y)\wedge h+a_3 h\wedge x
\end{array}\]
so
$a_1=-c_2+2a_1$ and $a_3=-a_2+a_3$. This is equivalent to $a_1=c_2$ and $a_2=0$.
Finally, $[h,y]=x+y$ and 
$\delta[h,y]=[\delta h,y]+[h,\delta y]$
imply
\[\begin{array}{rcl}
&&(a_1+b_1)x\wedge y+(a_2+b_2) y\wedge h+ (a_3+b_3)h\wedge x\\
&=&[c_1x\wedge y+c_2y\wedge h+c_3 h\wedge x,y]+
[h,b_1x\wedge y +b_2y\wedge h+b_3h\wedge x]
\\
&=&c_2y\wedge x+c_3 y\wedge x+
2b_1x\wedge y +b_2(x+y)\wedge h+b_3h\wedge x
\end{array}\]
So, $a_1+b_1=-c_2-c_3+2b_1$, $a_2+b_2=b_2$, $a_3+b_3=-b_2+b_3$.
Solving all the linear equations, we obtain
$a_2=a_3=b_2=0$, $c_2=a_1$, $c_3=b_1-2a_1$.
Hence, a general 1-cocycle $\delta$, is given by
\[\begin{array}{rcll}
\delta (x)&=&a_1x\wedge y\\
\delta (y)&=&b_1x\wedge y &+b_3h\wedge x\\
\delta (h)&=&c_1x\wedge y+a_1y\wedge h&+(b_1-2a_1)h\wedge x
\end{array}
\]
Or, in matrix notation:
\[
\delta=
\left(\begin{array}{ccc}
a_1&b_1&c_1\\
0    &0    &a_1\\
0    &b_3&b_1-2a_1
\end{array}
\right)
\]
The co-Jacobi condition for a general 1-cocycle is simply 
$2a_1^2=0$. Hence, a 1-cocycle satisfying also co-Jacobi is of the form
\[
\delta=
\left(\begin{array}{ccc}
0&b_1&c_1\\
0&0&0\\
0&b_3&b_1
\end{array}
\right)
\]
The Lie algebras automorphism group of $\r_3$ is the following subgroup of $\GL(3,\R)$
\[
\left\{\phi _{\mu,\rho,a,b}=
\left(
\begin{array}{ccc}
\mu&\rho&a\\
0&\mu&b\\
0&0&1
\end{array}
\right):\mu,\rho,a,b\in\R,\ \mu\neq 0
\right\}
\]
Under the action of the automorphism group, a 1-cocycle $\delta$ maps into
$%\[
\delta'=
\left(\begin{array}{ccc}
0 &\frac{b_1+b b_3}{\mu}&\frac{2bb_1+b^2b_3+c_1}{\mu^2}\\
0&0&0\\
0&b_3&\frac{b_1+bb_3}{\mu}
\end{array}\right)
$, %\]
then $b_3$ is an invariant.

\noindent {\em Case $b_3\neq 0$}.
Taking $b=-b_1/b_3$ we get $b_1'=0$, so we may assume $b_1=0$.
The conditions $b_1'=0$ is preserved only if $b=0$, and in that case $\delta$ changes into
\[
\delta'=
\left(\begin{array}{ccc}
0 &0&\frac{c_1}{\mu^2}\\
0&0&0\\
0&b_3&0
\end{array}\right)\]
So, $c_1$ can be chosen up to positive scalar and we can take the numbers
$c=0,\pm 1$  as representatives.
We conclude that the isomorphism classes of Lie bialgebras with coefficient $b_3\neq 0$ consist of three 1-parameter families with cobrackets:
\[
\delta _{b_3,c_1}=
\left(\begin{array}{ccc}
0 &0&c_1\\
0&0&0\\
0&b_3&0
\end{array}\right): c_1=0,1,-1,\ b_3\neq 0\]
For $c_1=0$, the automorphism group consists of
$\left\{\left(
\begin{array}{ccc}
\mu&\rho&a\\
0&\mu&0\\
0&0&1
\end{array}\right):\mu,a\in\R,\mu\neq 0\right\}
$, and for $c_1=\pm1$, we have $c_1'=c_3/\mu^2$, so  $\mu^2$ must be equal to 1, and the automorphism group is 
$\left\{\left(
\begin{array}{ccc}
\mu&\rho&a\\
0&\mu&0\\
0&0&1
\end{array}\right):a\in\R,\mu=\pm 1\right\}
$.

\

\noindent {\em Case $b_3=0$}.
We have then
$
\delta=
\left(\begin{array}{ccc}
0&b_1&c_1\\
0&0&0\\
0&0&b_1
\end{array}
\right)\mapsto
%\quad \hbox{ and }\quad
\delta'=
\left(\begin{array}{ccc}
0& b_1  /\mu&
(2 b b_1   + c_1)/\mu^2\\
0&0&0\\
0&0& b_1 /  \mu
\end{array}\right)
$
hence, $b_1$ is determined up to a multiple; 
we consider the cases $b_1\neq 0$ and  $b_1=0$.

\

\noindent \noindent {\em Case $b_1\neq 0$}.
We may assume $b_1=1$ then
$\delta'=
\left(\begin{array}{ccc}
0& 1  /\mu&
(2 b   + c_1)/\mu^2\\
0&0&0\\
0&0& 1 /  \mu
\end{array}\right)
$.
If we whish to preserve $b_1=1$ we need to impose that $\mu =1$; we obtain
$%\[
\delta'=
\left(\begin{array}{ccc}
0& 1  &
2 b   + c_1\\
0&0&0\\
0&0& 1 
\end{array}\right)
$. %\]
We may choose $b=-\frac {c_1}{2}$, so the new $c_1'=0$
 and take as a representative of the class of isomorphism
$%\[
\delta=
\left(\begin{array}{ccc}
0& 1  & 0\\
0&0&0\\
0&0& 1 
\end{array}\right)
$. %\]
The automorphisms group of this Lie bialgebra is
\[
G_{b_1\neq 0}=\left\{\left(
\begin{array}{ccc}
1&\rho&a\\
0&1&0\\
0&0&1
\end{array}\right):\rho,a\in\R,\right\}
\]

\noindent {\em Case $b_1= 0$}.
We have then
\[
\delta=
\left(\begin{array}{ccc}
0&0&c_1\\
0&0&0\\
0&0&0
\end{array}
\right)\mapsto
%\quad \hbox{ and }\quad
\delta'=
\left(\begin{array}{ccc}
0& 0 & c_1/\mu ^2\\
0&0&0\\
0&0& 0 
\end{array}\right)
\]
Hence, we obtain that $c_1=0,\pm 1$ are all the possibilities for $c_1$.
The automorphisms group of such Lie bialgebra with $c_1=0$ is
\[
G_{b_1= c_1=0}=\left\{\left(
\begin{array}{ccc}
\mu &\rho&a\\
0&\mu &b\\
0&0&1
\end{array}\right):\mu,\rho,a,b\in\R,\mu\neq 0\right\}
\]

On the other hand, the automorphisms group of the Lie bialgebras classes with $c_1=\pm 1$ 
consists of the  Lie algebra maps satisfying $\mu^2=1$; hence
\[
G_{b_1= 0,c_1\neq 0}=\left\{\left(
\begin{array}{ccc}
\mu &\rho&a\\
0&\mu &b\\
0&0&1
\end{array}\right):\mu,\rho,a,b\in\R,\mu =\pm 1\right\}
\]

\begin{theorem}
The isomorphism classes of Lie bialgebra structures on $\r_3$ is given by the following list of cobrackets:
\[
\delta _{b_1\neq 0,b_3=0}=
\left(\begin{array}{ccc}
0& 1&0\\
0&0&0\\
0&0& 1
\end{array}
\right)\ \hbox{ and }
\delta _{c_1,b_3}=
\left(\begin{array}{ccc}
0&0& c_1\\
0&0&0\\
0&b_3&0
\end{array}\right): b_3\in\R,\ c_1=0,\pm 1
\]
\end{theorem}

%%%%%%%%%%%%%COMIENZO INSERTADO
%%%%%%%%%%%%%%%%%%%%%%
%%%%%%%%%%%%%%%%%%%%%%%%%%%%%%%%%

%%%%%%%%%%%%%%%%%%%%%%%%%%%%%%%%%%
\section{Lie bialgebra structures on $\r_{3,\lambda }$ , with $|\lambda |\leq 1$.}
%%%%%%%%%%%%%%%%%%%%%%%%%%%%%%%%%%

Recall that $\r _{3,\lambda }$ is the Lie algebra with bases $\{x,\ y, \ h\}$, and bracket
$[h,x]=x$, $[h,y]=\lambda y$,  $[x,y]=0$. 
We list general properties for  $\r_{3,\lambda}$:
\begin{itemize}
%\item if $\lambda\neq 0$, then
%$[\g,\g]=\R x\oplus\R y=[\g,[\g,\g]]$. 
%\item $[[\g,\g],[\g,\g]]=0$.
\item if $\lambda\neq 0$ then $\Z (\g)=0$; if $\lambda=0$ then $\Z(\g)=\cl{y}$.
\item \label{lambda2r3} If $\lambda = -1$ then $ \Lambda ^2(\g)^{\g}=\cl{x \w y}$.
If $\lambda \neq -1$ then $\Lambda ^2(\g)^{\g}=0$.
\end{itemize}

%\begin{lemma}\label{lambda2r3}
%There are no non-trivial invariants in $\Lambda ^2(\r _{3,\lambda })$ in the case
% $\lambda \neq -1$.
%If $\lambda = -1$ then $ x \w y\in\Lambda ^2(\g)^{\g}$.
%\end{lemma}
%
%\begin{proof}$\bullet$ \noindent {\em Case $\lambda \neq 0$}.
%Let $\om =ax\wedge y+by\wedge h+ch\wedge x\in\Lambda ^2(\g)^{\g}$ then
%\[
%\ad_h(\om) =\ad_h(ax\wedge y+by\wedge h+ch\wedge x)
 %                    =a(1+\lambda)x\wedge y+b\lambda y\wedge h+ch\wedge x=0
%\]
%implies $c=0$, $a(1+\lambda )=0$ and $b\lambda =0$, 
%where the last condition implies $b=0$ since $\lambda\neq 0$. 
%
%\[
%\ad_x(\om) =\ad_x(ax\wedge y+by\wedge h+ch\wedge x)
%                     =-b y\wedge h=0
%\]
%implies again $b=0$.
%
%\[
%\ad_y(\om) =\ad_y(ax\wedge y+by\wedge h)=0
%\]
% Hence, $ x \w y$ is invariant if and only if $\lambda =-1$.
%
%\
%
%$\bullet$ \noindent {\em Case $\lambda =0$}.
%Let $\om =ax\wedge y+by\wedge h+ch\wedge x\in\Lambda ^2(\g)^{\g}$ then
%\[
%\ad_h(\om) =\ad_h(ax\wedge y+by\wedge h+ch\wedge x)
%                     =ax\wedge y+ch\wedge x=0
%\]
%implies $c=0$ and $a=0$.
%
%\[
%\ad_x(\om) =\ad_x(by\wedge h)
%                     =-b y\wedge h=0
%\]
%implies $b=0$.
%
% Hence, $\Lambda ^2(\g)^{\g}=0$ also in the case $\lambda =0$.
%\end{proof}
%

%%%%%%%%%%%%%%%%%
%\subsubsection{ 1-cocycle condition}
%%%%%%%%%%%%%%%%%%%%%
\begin{proposition}
All the 1-cocycles on the Lie algebra  $\r_{3,\lambda }$  with $|\lambda |\leq 1$
are
\[
\delta =
\left(\begin{array}{ccc}
a_1&\lambda c_3&c_1\\
0&\lambda a_3&\lambda a_1\\
a_3&0&c_3
\end{array}\right)\hbox{if $\lambda\neq 1$,}
\hskip2cm
\delta =
\left(\begin{array}{ccc}
a_1&c_3&c_1\\
a_2& a_3& a_1\\
a_3&b_3&c_3
\end{array}\right)\hbox{if $\lambda=1$}
\]
where we consider the basis 
$\{x\wedge y,\ y\wedge h, \ h\wedge x\}$ of $\Lambda ^2 (\g)$ and notations as in  \ref{gcJc}.
\end{proposition}
\begin{proof} 
%Consider the basis  $\{x\wedge y,\ y\wedge h, \ h\wedge x\}$ of $\Lambda ^2 (\g)$ and %notations as in  \ref{gcJc}.
Let $\delta :\g\to\Lambda ^2 (\g)$ be a 1-cocyle, then
$\delta[h,x]=[\delta h,x]+[h,\delta x]$ and $[h,x]=x$ imply
\[\begin{array}{rl}
&a_1x\wedge y+a_2y\wedge h+a_3h\wedge x\\
=&
[c_1x\wedge y+c_2y\wedge h+c_3h\wedge x,x]+
[h,a_1x\wedge y+a_2y\wedge h+a_3h\wedge x]
\\
	=&c_2y\wedge x+a_1(1+\lambda )x\wedge y+\lambda a_2y\wedge h+a_3h\wedge x
\end{array}
\]
We conclude $\lambda a_1=c_2$, $a_2=\lambda a_2$, so $a_2=0$ if $\lambda\neq 1$, and no condition in $a_2$ for $\lambda=1$.
In an analogous way, using the cocycle condition $\delta[h,y]=[\delta h,y]+[h,\delta y]$
for  $[h,y]=\lambda y$, 
we get
\[\begin{array}{rl}
&\lambda (b_1x\wedge y+b_2y\wedge h+b_3h\wedge x)
\\
&=
[c_1x\wedge y+c_2y\wedge h+c_3h\wedge x,y]+
[h,b_1x\wedge y+b_2y\wedge h+b_3h\wedge x]
\\
&=\lambda c_3y\wedge x+(1+\lambda )b_1x\wedge y+\lambda b_2y\wedge h+b_3h\wedge x
\end{array}\]
so $b_1=\lambda c_3$ and $b_3=\lambda b_3$, so again $b_3=0$ if $\lambda\neq 1$ and 
no restriction on $b_3$ for $\lambda=1$.
The third condition is $\delta[x,y]=[\delta x,y]+[x,\delta y]$; since $[x,y]=0$, we get
\[
0=[a_1x\wedge y+a_2y\wedge h+a_3h\wedge x,y]+
[x,b_1x\wedge y+b_2y\wedge h+b_3h\wedge x]
=\lambda a_3y\wedge x -b_2y\wedge x
\]
so  $b_2=\lambda a_3$. 
As a consequence, the general form of a 1-cocycle is, for $\lambda\neq 1$:
\[\begin{array}{rcccccl}
\delta (x)&=&a_1x\wedge y&& &+&  a_3h\wedge x\\
\delta (y)&=&\lambda c_3x\wedge y &+&\lambda a_3y\wedge h\\
\delta (h)&=&c_1x\wedge y&+&\lambda a_1y\wedge h&+&c_3h\wedge x
\end{array}
\]
and for $\lambda=1$:
\[\begin{array}{rcccccl}
\delta (x)&=&a_1x\wedge y&+&a_2y\w h&+& a_3h\wedge x\\
\delta (y)&=& c_3x\wedge y &+&a_3y\wedge h&+&b_3 h\w x\\
\delta (h)&=&c_1x\wedge y&+&a_1y\wedge h&+&c_3h\wedge x
\end{array}
\]
%In matrix notation:
%\[
%\delta =
%\left(\begin{array}{ccc}
%a_1&\lambda c_3&c_1\\
%0&\lambda a_3&\lambda a_1\\
%a_3&0&c_3
%\end{array}\right)\hbox{if $\lambda\neq 1$,}
%\hskip2cm
%\delta =
%\left(\begin{array}{ccc}
%a_1&c_3&c_1\\
%a_2& a_3& a_1\\
%a_3&b_3&c_3
%\end{array}\right)\hbox{if $\lambda=1$}
%\]
\end{proof}

\noindent \noindent {\em Case  $\r _{3,\lambda }$ with $|\lambda |\leq 1, \lambda \neq  \pm 1$.}

%%%%%%%%%%%%%%%%%%%%
%\subsubsection{Isomorphism classes}
%%%%%%%%%%%%%%%%%%%

 \begin{proposition}The automorphism group of the Lie algebra $\r_{3,\lambda}$ with   $\lambda \neq \pm1$,
is the following subgroup of $\GL (3,\R)$:
\[\Aut(\g)=\left\{  
\left(\begin{array}{ccc}
\mu &0 & a\\
0 &\nu &b\\
0 &0 & 1
\end{array}
\right):\ \mu \nu \neq 0 
\right \}
\]
\end{proposition}
\begin{proof}
Let us see that if $\phi:\g\to\g$ is an automorphism of  Lie algebra,
then
\[\phi(x)=\mu x;\ \phi(y)=\nu y;\ \phi(h)=h+a x+b y
\]
for some $\mu,\nu ,a,b\in\R,\ \mu,\nu \neq 0$. Actually, the elements $x$ and $y$ may be characterized, up to scalar multiple,
as the generators of
 $[\g, \g]$ and eigenvectors of $\ad_z$, for all $z\in\g\setminus[\g,\g]$.
Moreover, given such  $z$, the element $x$ distinguishes from $y$ as being the eigenvector
corresponding to the eigenvalue with smaller absolute value.
 Explicitely, if $z\notin[\g,\g]$, $z=ch+ax+by$ with $c\neq 0$, 
\[
\ad _z(x)=[ch+ax+by, x]=cx,\quad
\ad _z(y)=[ch+ax+by, y]=c\lambda y
\]
where $|\lambda|<1$. This implies  $\phi(x)=\mu x $ and $\phi(y)=\nu y$, for some $\mu,\ \nu\neq 0$.

Let $\phi (x)=\tilde{x}=\mu x$,  $\phi (y)=\tilde{y}=\nu y$,  $\phi (h)=\tilde{h}=ch+ax+by$; 
if $\phi$ is a Lie algebra morphism, then
 $[\tilde h,\tilde{x}]=\tilde{x}$, so $c=1$. 
\end{proof}

Let $\phi$ be as before, and let
$\delta ':=(\phi\wedge\phi)^{-1}\delta\phi$; explicitly
\[
\delta '=
\left(\begin{array}{ccc}
\frac{a_1+ a_3b}{\nu}&\lambda \frac{c_3+ aa_3}{\mu}
& \frac{c_1 
 +a (a_1+a_3 b) (1+\lambda)+b c_3 (1+\lambda)}{\mu\nu}
 \\
0&\lambda  a_3 &  \lambda  \frac{a_1+ a_3b}{\nu}
\\
 a_3&0&\frac{c_3+aa_3   }{\mu}\end{array}\right)
\]
Notice that, if $a_3\neq 0$, by means of an automorphism with $b=-a_1/a_3$ and $a=-c_3/a_3$,
we get $\delta'$ with $a_1'=0=c_3'$; explicitly
$\delta '=
\left(\begin{array}{ccc}
0& 0
& \frac{a_3c_1 -a_1c_3 (1+\lambda)}{\mu\nu}
 \\
0&\lambda a_3 & 0\\
 a_3&0&0\end{array}\right)
=
\left(\begin{array}{ccc}
0& 0
& c_1'
 \\
0&\lambda a_3 & 0\\
 a_3&0&0\end{array}\right)
$.

%%%%%%%%%%%%%%%%%%%
%\subsubsection{Co-Jacobi condition }
%%%%%%%%%%%%%%%%%%%%
\

Co-Jacobi condition for $\delta$ (recall $\lambda \neq \pm 1$) is
$(1-\lambda)(a_3c_1-a_1c_3(1+\lambda))=0$, or equivalently
$a_3c_1-a_1c_3(1+\lambda))=0$.
If $a_3=0$, this condition reduces to $a_1b_1=0$.
Note that, up to isomorphism,
we may independently change  $a_1$ and $b_1$ by $a_1'=a_1/\nu$ and
$c_3'=c_3/\mu$ respectively. But also, because one of them is zero, 
we have the following possibilities:
\begin{itemize}
\item
$(a_1,c_3)=(0,0)$, $c_1=0$ or $1$ because $c_1$ is determined (up to isomorphism) up to scalar multiple.
\item
$(a_1,c_3)=(1,0)$, and $c_1$ changes into
 $c_1'=(c_1+a(1+\lambda))/\mu$ (we need $\nu=1$ in order to preserve $a_1'=1$),
we see that we can  choose $a$ such that $c'_1=0$.
\item
$(a_1,c_3)=(0,1)$, and $c_1$ changes into $c_1'=(c_1+b(1+\lambda))/\nu$,
so we can also choose $c_1=0$.
\end{itemize}
If $a_3\neq 0$, we may assume $a_1=0=c_3$, then
%use an automorphism changing $\delta$ into $\delta'$ with $a_1'=0=c_3'$ (see remark before co-Jacobi equation). With this new cobracket, 
co-Jacobi implies $c_1=0$. Hence, every cocycle with $a_3\neq 0$ satisfying co-Jacobi
is equivalent to
\[
\left\{
\delta _{a_3}=
\left(\begin{array}{ccc}
0 & 0 &0 \\
0&a_3 \lambda & 0 \\
a_3&0&0
\end{array}\right):0\neq a_3\in\R
\right\}
\]
The parameter
$a_3$ can not be modified using a Lie algebra automorphism, it is an invariant;
we have a 1-parameter family of isomorphism classes, parametrized by $a_3$.

\noindent 
The  bialgebra automorphisms are of the form
$
\phi =
\left(\begin{array}{ccc}
\mu & 0 &0\\
0& \nu & 0 \\
 0&0&1
\end{array}\right)
:\mu,\ \nu\neq 0
$.
We have finished the proof of the next statement.
%\subsubsection{Summary for $\g=\r_{3,\lambda} : \lambda\neq \pm1$}

\begin{theorem}
The set of representatives of all isomorphisms classes of Lie bialgebras on
$\r_{3,\lambda} : \lambda\neq \pm1$ is given by the following cobrackets
\[\left\{
\left(
\begin{array}{ccc}
0&0&1\\
0&0&0\\
 0&0&0
\end{array}
\right);
\quad
\left(
\begin{array}{ccc}
1&0&0\\
0&0&\lambda\\
 0&0&0
\end{array}
\right);
\quad
\left(
\begin{array}{ccc}
0&\lambda&0\\
0&0&0\\
0&0&1
\end{array}
\right)\
\hbox{and}
\left(
\begin{array}{ccc}
0&0&0\\
0&\lambda a_3&0\\
a_3&0&0
\end{array}
\right)
:a_3\in\R
 \right\}
\]
\end{theorem}

%%%%%%%%%%%
%%%%%%%%%%%%%%%%%%%

\noindent{\em Case  $\g=\r_{3,\lambda}$ with $\lambda=-1$}. 
Let us recall that $\r _{3,\lambda =-1}$ is the Lie algebra with bases $\{x,\ y, \ h\}$, and bracket
$[h,x]=x$, $[h,y]=- y$,  $[x,y]=0$.
%\subsubsection{Automorphisms}
The automorphisms group of $\r_{3,\lambda=-1}$, % with  $\lambda =-1$,
in the ordered basis $\{x,\ y, \ h\}$, identifies with the following subgroup of $\GL (3,\R)$:
\[\Aut(\r_{3,\lambda=-1})=\left\langle  
\phi _0=\left(\begin{array}{ccc}
0 &1 & 0\\
1 &0&0\\
0 &0 & -1
\end{array}
\right),\ 
\phi _{\mu, \nu,a,b}=
\left(\begin{array}{ccc}
\mu &0 & a\\
0 &\nu &b\\
0 &0 & 1
\end{array}
\right):\ \mu \nu \neq 0
\right\rangle
\]
In fact, $\phi _{\mu, \nu,a,b}$ is an automorphism by the same reasons as above, but in this case, the absolute value of the eigenvalues of $x$ and $y$ are the same. Actually,
\[
x\mapsto y;\quad y\mapsto x;\quad h\mapsto -h
\]
is an automorphism, which we denote by $\phi_0$. Moreover, any automorphism is obtained by compositions of the above ones.

%\subsubsection{1-Cocycle condition}
The set of 1-cocycles was computed for any $\lambda$ but for convenience  in case $\lambda = -1$ we parametrize them by
$%\[
\delta =
\left(\begin{array}{ccc}
a_1&b_1&c_1\\
0&- a_3&- a_1\\
a_3&0&-b_1
\end{array}\right)
$ %\]
instead of $
\delta =
\left(\begin{array}{ccc}
a_1&\lambda c_3&c_1\\
0&\lambda a_3&\lambda a_1\\
a_3&0&c_3
\end{array}\right)$.
For these cocycles, co-Jacobi reads $2a_3c_1=0$.
The action of the automorphism group is
\[
\xymatrix@+2.5pc{
 \delta =
\left(\hbox{$
\begin{array}{ccc}
a_1&b_1&c_1\\
0&- a_3&- a_1\\
a_3&0&-b_1
\end{array}$}\right)
\ar@{|->}[r]^{\phi _{\mu, \nu,a,b}}
&
\delta '=
\left(\hbox{$
\begin{array}{ccc}
\frac{a_1+a_3b}{\nu}&\frac{b_1-aa_3}{\mu}& \frac{c_1}{\mu\nu}           
 \\
0&-a_3&\frac{-a_1-a_3b}{\nu}\\
a_3&0&\frac{-b_1+aa_3}{\mu}
\end{array}$}
\right)}
\]\[
\xymatrix@+2.5pc{
 \delta =
\left(\hbox{$
\begin{array}{ccc}
a_1&b_1&c_1\\
0&- a_3&- a_1\\
a_3&0&-b_1
\end{array}$}\right)
\ar@{|->}[r]^{\phi _0}
&
\delta '=
\left(\hbox{$
\begin{array}{ccc}
-b_1 & -a_1 & c_1 \\
0& a_3 &b_1 \\
-a_3 & 0 &-b_1
\end{array}$}
\right)
}
\]

%We distinguish  the following  isomorphism classes:

\noindent {\em Case $a_3\neq 0$:} Co-Jacobi implies $c_1=0$. But also,
taking an automorphism $\phi$ with  parameters $a=b_1/a_3$ and
$b=-a_1/a_3$ we get $\delta '$ with $a_1'=b_1'=0$.
Besides, using $\phi _0$, $a_3\mapsto a_3'=-a_3$, so
we may choose $a_3>0$.
We conclude that inside this isomorphism class, we have the representative
\[
\delta=\left(\begin{array}{ccc}
0&0&0\\
0&-a_3&0\\
a_3&0&0
\end{array}\right):a_3>0
\]

\noindent {\em Case $a_3=0$:} In this case, co-Jacobi condition is automatically satisfied.
For each 3-uple $(a_1,b_1,c_1)$ we have
$\delta_{a_1,b_1,c_1} \cong \delta_{\frac{a_1}{\mu},\frac{b_1}{\nu},\frac{c_1}{\mu\nu}}$ \ie
\[
 \delta_{a_1,b_1,c_1} =
\left(\begin{array}{ccc}
a_1&b_1&c_1\\
0&0&- a_1\\
0&0&-b_1
\end{array}\right)
\cong
\delta_{\frac{a_1}{\mu},\frac{b_1}{\nu},\frac{c_1}{\mu\nu}}=
\left(\begin{array}{ccc}
\frac{a_1}{\nu}&\frac{b_1}{\mu}& \frac{c_1}{\mu\nu}           
 \\
0&0&\frac{-a_1}{\nu}\\
0&0&\frac{-b_1}{\mu}
\end{array}\right)
\]
By means of the isomorphism $\phi_0$
%=\left(\begin{array}{ccc}
%0&1&0\\
%1&0&0\\
%0&0&-1
%\end{array}\right)$
 we obtain aditionally
$ \delta_{a_1,b_1,c_1} \cong
 \delta_{-b_1,-a_1,c_1}$.
Choosing conveniently  $\mu$ and $\nu$, we arrive at the following
list of iso classes:
\begin{eqnarray*}
\delta_{0,0,0}=0;\quad
\delta_{0,0,1}=&
\left(\begin{array}{ccc}
0&0&1\\
0&0&0\\
0&0&0
\end{array}\right);\quad
\delta_{1,0,0}&=\left(\begin{array}{ccc}
1&0&0\\
0&0&-1\\
0&0&0
\end{array}\right);
\\
\delta_{1,0,1}=&\left(\begin{array}{ccc}
1&0&1\\
0&0&-1\\
0&0&0
\end{array}\right)
;\quad
\delta _{1,1,c_1}&=\left(\begin{array}{ccr}
1&1&c_1\\
0&0&-1\\
0&0&-1
\end{array}\right)
\end{eqnarray*}
where we identify $\delta _{1,0,0}\cong\delta _{0,1,0}$; $\delta _{1,0,1}\cong\delta _{0,1,1}$.

%%%%%%%%%%%%%%%%%FIN INSERTADO%%%%%%%%%%%%%%%%
%%%%%%%%%%%%%%%%%%%%%%%%%%%%%%%%

%%%%%%%%%%%%%REVISAR ESTA ANTES%%%%%%%%%%%%

%%%%%%%%%%%%%%%%%%%
%%%HASTA ACA   REPETIDO%%%%%%%%%%%%%%%%%%%%%%%%%%%

\section{Lie bialgebra structures on  $\r _{3,\lambda }$ with $\lambda = 1$.}

Recall that $\r _{3, 1 }$ is the Lie algebra with ordered
bases $\{x,\ y, \ h\}$ and bracket determined by
$[h,x]=x$, $[h,y]= y$,  $[x,y]=0$.
It can be easily verified that the automorphism group is the  subgroup
of $\GL (3,\R)$ expressed as matrices as:
\[
 \Aut(\r_{3,1})=\left\{  \phi_{\mu ,\nu,\rho,\sigma}^{ a ,b}=
\left(\begin{array}{ccc}
\mu &\rho & a\\
\sigma &\nu &b\\
0 &0 & 1
\end{array}
\right):\ \mu \nu -\rho\sigma\neq 0 
\right \}
\]
%\begin{lemma}There are no no-trivial invariants in $\Lambda ^2(\r _{3,\lambda })$ in the case $\lambda = 1$.
%\end{lemma}
%\begin{proof}Let $\om =ax\wedge y+by\wedge h+ch\wedge x\in\Lambda ^2(\g)^{\g}$ then
%\[
%\ad_h(\om) =\ad_h(ax\wedge y+by\wedge h+ch\wedge x)
%                     =2ax\wedge y+by\wedge h+ch\wedge x=0
%\]
%implies $a=0=b=c$. Hence, $\Lambda ^2(\g)^{\g}=0$.
%\end{proof}
%
%\subsection{1-cocycle condition}
Recall that the 1-cocycles are  
%\[\begin{array}{rcl}
%\delta (x)&=&a_1x\wedge y+a_2y\wedge h+ a_3h\wedge x\\
%\delta (y)&=&b_1x\wedge y + a_3y\wedge h+b_3h\wedge x\\
%\delta (h)&=&c_1x\wedge y+ a_1y\wedge h+b_1h\wedge x
%\end{array}
%\]
in matrix notation given by:
$%\[
\delta =
\left(\begin{array}{ccc}
a_1&b_1&c_1\\
a_2& a_3& a_1\\
a_3&b_3&b_1
\end{array}\right)
$%\]
with $a_1,a_2, a_3,b_1,b_3,c_1\in\R$.
The co-Jacobi identity is always satisfied.

%%%%%%%%%%%%%%COMIENZO TEOREMA Y TABLA LAMBDA=1
%%%%%%%%%%%%%%%
\begin{theorem}
For the Lie algebra $\g=\r_{3,\lambda=1}$ the exhaustive list
of representatives of the isomorphism classes of Lie bialgebras, or equivalently,
the exhaustive list of a representative set of cobrackets is the following :
%(we write the matrices of the corresponding $\delta$'s):
\end{theorem}

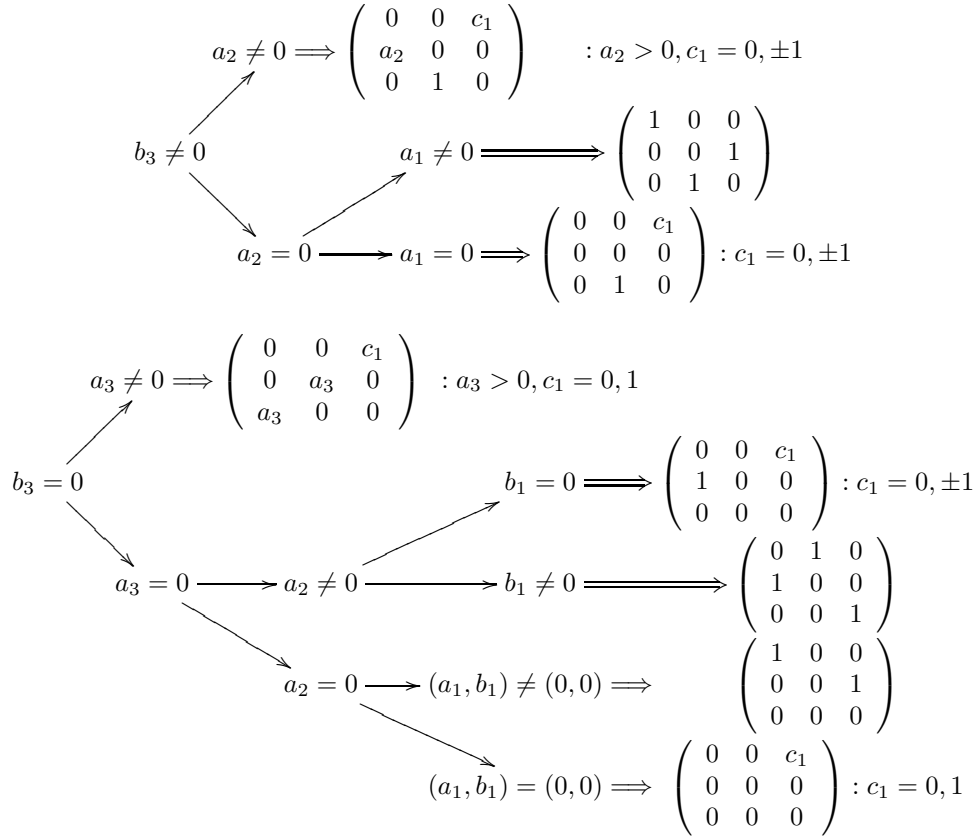
\[
\xymatrix@-2.3pc{
&a_2\neq 0\TO&
\left(\hbox{$
\begin{array}{ccc}
0&0&c_1\\%=0,\pm 1\\
a_2&0&0\\
0&1&0
\end{array}
$}\right)&: a_2>0, c_1=0,\pm 1
\\
b_3\neq 0\ar[rd]\ar[ru]&&a_1\neq 0\ar@{=>}[r]
&
\left(\hbox{$
\begin{array}{ccc}
1&0&0\\
0&0&1\\
0&1&0\end{array}$}\right)
\\
&a_2=0\ar[ru]\ar[r]&a_1=0\ar@{=>}[r]&
\left(\hbox{$
\begin{array}{ccc}
0&0&c_1\\%0,\pm 1\\
0&0&0\\
0&1&0\end{array}$}
\right):c_1=0,\pm 1
}\]
\[
\xymatrix@-2.3pc{
&a_3\neq 0\TO&
\left(\hbox{$
\begin{array}{ccc}
0&0&c_1\\%=0,1\\
0&a_3&0\\
a_3&0&0\end{array}$}
\right)&: a_3>0,c_1=0,1\\
b_3=0\ar[ru]\ar[rd]&&&b_1=0\ar@{=>}[r]&
\left(\hbox{$
\begin{array}{ccc}
0&0&c_1\\%0,\pm 1\\
1&0&0\\
0&0&0\end{array}$}
\right):c_1=0,\pm 1
\\
&a_3=0\ar[r]\ar[rd]&a_2\neq 0\ar[r]\ar[ru]&b_1\neq 0\ar@{=>}[r]&
\left(\hbox{$
\begin{array}{ccc}
0&1&0\\
1&0&0\\
0&0&1\end{array}$}
\right)
\\
&&a_2=0\ar[r]\ar[rd]&(a_1,b_1)\neq(0,0)\TO&
\left(\hbox{$
\begin{array}{ccc}
1&0&0\\
0&0&1\\
0&0&0\end{array}$}
\right)
\\
&&&(a_1,b_1)=(0,0)\TO&
\left(\hbox{$
\begin{array}{ccc}
0&0&c_1\\
0&0&0\\
0&0&0\end{array}$}
\right):c_1=0,1}
\]
%%%%%%%%%%%%%%FIN TABLA LAMBDA=1%%%%%%%%%%%%%
%%%%%%%%%%%%%%%%%%%%%%%%%%%%%%%%
We dedicate the rest of the section to the proof of this result, which proceeds in the cases
given by the previous diagram.

\noindent {\bf Action of the automorphism group}.
%If %$\phi$ be an automorphism of the Lie algebra, 
$\delta '=(\phi\wedge\phi)^{-1}\delta\phi$,   $\phi=\phi_{\mu ,\nu,\rho ,\sigma}^{a, b}$, then
$\delta '=$\[
\left(
\begin{array}{ccc}
\frac{\mu(a_1 +a a_2 +a_3 b)+ \sigma(a a_3+b_1+b b_3)}{\mu \nu-\rho \sigma}&
\frac{\nu(a a_3 +b_1 +b b_3)+\rho(a_1 +a a_2 +a_3 b)}{\mu \nu-\rho \sigma}&
\frac{a(aa_2+2 a_1+2 a_3 b)+2 b b_1+b^2b_3+c_1}{\mu\nu-\rho\sigma}\\
\frac{a_2 \mu^2+\sigma (2 a_3 \mu+b_3 \sigma)}{\mu \nu-\rho \sigma}&
\frac{a_3 \mu \nu+a_2 \mu \rho+b_3 \nu \sigma+a_3 \rho \sigma}{\mu \nu-\rho \sigma}&
\frac{\mu(a_1 +a a_2 +a_3 b)+ \sigma(a a_3+b_1+b b_3)}{\mu \nu-\rho \sigma}\\
\frac{a_3 \mu \nu+a_2 \mu \rho+b_3 \nu \sigma+a_3 \rho \sigma}{\mu \nu-\rho \sigma}&
\frac{b_3 \nu^2+\rho (2 a_3 \nu+a_2 \rho)}{\mu \nu-\rho \sigma}&
\frac{\nu(a a_3 +b_1 +b b_3)+\rho(a_1 +a a_2 +a_3 b)}{\mu \nu-\rho \sigma}
\end{array}
\right)
\]
Considering the special type of automorphisms with $\rho=0=\sigma$, we obtain
\[\delta '=
\left(
\begin{array}{ccc}
\frac{a_1+a a_2 +a_3 b}{ \nu}&
\frac{a a_3 +b_1 +b b_3 }{\mu }&
\frac{a^2a_2+a (2 a_1+2 a_3 b)+2 b b_1+b^2b_3+c_1}{\mu\nu}\\
\frac{a_2 \mu}{ \nu}&
a_3&
\frac{a_1 +a a_2 +a_3 b }{ \nu}\\
a_3 &
\frac{b_3 \nu}{\mu }&
\frac{a a_3 +b_1 +b b_3 }{\mu }
\end{array}
\right)
\]
If also $\mu=\nu=1$ and $a=0$, we get
\[\delta '=
\left(
\begin{array}{ccc}
a_1 +a_3 b&
b_1 +b b_3&
2 b b_1+b^2b_3+c_1\\
a_2 &
a_3&
a_1 +a_3 b \\
a_3 &
b_3&
b_1 +b b_3 
\end{array}
\right)
\]
%%%%%%%%%%%%%%%%%
\noindent{\em Case  $b_3\neq 0$}.
%%%%%%%%%%%%%%%%%%
From the previous analysis, we see that if $b_3\neq 0$, we can choose
$b$ such that  $b_1'$
transforms into zero. Explicitly, take
$b=-b_1/b_3$. So we may assume from the beginning, %without lost of generality,
 that $b_1=0$.
Consider now $\delta$ of type 
$
\delta =
\left(\begin{array}{ccc}
a_1&0&c_1\\
a_2& a_3&a_1\\
a_3&b_3&0
\end{array}\right)
$.
If we use an automorphism with  $\rho=0$ we get
$%\[
\delta '=
\left(
\begin{array}{ccc}
\frac{\mu(a_1+a a_2 +a_3 b)+\sigma(a a_3+b b_3)}{\mu \nu}&\frac{a a_3+b b_3}{\mu}&*\\
*  & * & *\\
* & * & *
\end{array}
\right)
$. %\]
Then we can preserve the condition $b_1'=0$ if we set for example $a=b_3$, $b=-a_3$. In this case, 
\[\delta'=\frac{1}{\mu\nu}
\left(
\begin{array}{ccc}
(a_1-a_3^2+a_2b_3)\mu&0&2 a_1 b_3-a_3^2b_3+a_2b_3^2+c_1\\
a_2 \mu^2+\sigma(2a_3  \mu+b_3 \sigma)&\nu(a_3 \mu +b_3  \sigma)&(a_1-a_3^2+a_2b_3)\mu\\
\nu(a_3 \mu +b_3  \sigma)&b_3 \nu^2&0\\
\end{array}
\right)
\]
Since $b_3\neq 0$, we may choose $\sigma=-a_3\mu/b_3$ to get $\delta '$
with $a_3'=0$, so we may assume  from the beginning that $a_3=0$, \ie
$%\[
\delta=
\left(
\begin{array}{ccc}
a_1&0&c_1\\
a_2&0&a_1\\
0&b_3&0
\end{array}
\right)
$. %\]
Using an automorphism  with $b=\rho=\sigma=0$,  $\nu=1$, we get
$%\[
\delta'=
\left(
\begin{array}{ccc}
a_1+aa_2&0&\frac{2aa_1+a^2a_2+c_1}{\mu}\\
a_2\mu & 0 & a_1+a_2\\
0&\frac{b_3}{\mu}&0
\end{array}
\right)
$. %\]
Notice that we can make $b_3'=1$ by means of $\mu =b_3$
so, we may assume $b_3=1$.
We distinguish two subcases: $a_2\neq 0$ and $a_2=0$.

%%%%%%%%%%%%%%%%%%%%%
%\subsubsection{$a_2\neq 0$}
%%%%%%%%%%%%%%%%%%
\noindent {\em If $a_2\neq 0$}, 
we can choose $a=-a_1/a_2$ and get $a_1'=0$, so we start from the beginning with% withmay suppose $a_1=0$; let us consider
\[
\delta=
\left(
\begin{array}{ccc}
0&0&c_1\\
a_2&0&0\\
0&1&0
\end{array}
\right)
\]
Now, with a general automorphism, 
\[
\delta'=
\frac{1}{\mu\nu-\rho\sigma}\left(
\begin{array}{ccc}
a     a_2 \mu + b \sigma & b \nu + a a_2 \rho& a^2 a_2  + c_1\\
      a_2 \mu^2 + \sigma^2& a_2 \mu 
        \rho + \nu \sigma& a a_2 \mu + b \sigma\\
      a_2 \mu \rho + \nu \sigma& \nu^2 + a_2 \rho^2& b \nu + a a_2 \rho
      \end{array}
\right)
\]
In order to preserve the conditions $a'_1=b'_1=a'_3=0$ we must solve the equations
\begin{eqnarray}
a     a_2 \mu + b \sigma=0 \\ 
b \nu + a a_2 \rho=0\\ 
a_2 \mu \rho + \nu \sigma=0
\end{eqnarray} 
Computing $(1)\nu-(2)\sigma$ we get
$
a     a_2 (\mu \nu-  \rho\sigma)=0
$.
Since $a_2\neq 0\neq \mu\nu-\rho\sigma$, we conclude $a=0$.
Going back to $\delta'$ but with $a=0$ we have
\[
\delta'=
\frac{1}{\mu\nu-\rho\sigma}\left(
\begin{array}{ccc}
 b \sigma & b \nu &  c_1\\
      a_2 \mu^2 + \sigma^2& a_2 \mu 
        \rho + \nu \sigma&  b \sigma\\
      a_2 \mu \rho + \nu \sigma& \nu^2 + a_2 \rho^2& b \nu 
      \end{array}
\right)
\]
Since $\nu$ and $\sigma$ can not be simultaneously zero, it must be $b=0$,
explicitly
\[
\delta'=
\frac{1}{\mu\nu-\rho\sigma}\left(
\begin{array}{ccc}
 0&0 &  c_1\\
      a_2 \mu^2 + \sigma^2& a_2 \mu 
        \rho + \nu \sigma& 0\\
      a_2 \mu \rho + \nu \sigma& \nu^2 + a_2 \rho^2&0
      \end{array}
\right)
\]
If $a_2<0$ then  there exists an automorphism with convenient $\mu$ and $\sigma$
such that $a_2'=0$, but this case will be considered later, so  assume $a_2>0$.
The condition $a_3'=0$ means
$a_2 \mu \rho + \nu \sigma=0$.
If $\mu\neq 0$, we can solve $\rho=-\nu\sigma/(a_2\mu)$ and get
\[
\delta'=
\left(
\begin{array}{ccc}
 0&0 & \frac{a_2c_1\mu}{\nu(a_2\mu^2+\sigma^2)}\\
\frac{      a_2 \mu}{\nu}&0&0\\
      0&\frac{ \nu}{\mu}&0
      \end{array}
\right)
\]
In order to get $b_3'=1$ we need $\mu=\nu$, so
$%\[
\delta'=
\left(
\begin{array}{ccc}
 0&0 & \frac{a_2c_1}{a_2\mu^2+\sigma^2}\\
      a_2 &0&0\\
      0& 1&0
      \end{array}
\right)
$. %\]

\noindent It is clear that applying automorphisms with $\mu\neq 0$, $a_2$ is an invariant
and $c_1$ may be chosen up to positive scalar.
If $\mu=0$,  we get
$%\[
\delta'=\frac{1}{-\rho \sigma}
\left(
\begin{array}{ccc}
 0&0 & c_1\\
  \sigma^2&\nu \sigma &0\\
    \nu \sigma & \nu^2+a_2\rho^2&0
      \end{array}
\right)
$. %\]

\noindent If we want to preserve $a_2'>0$ and $b_2'=0$ we need $\sigma\neq 0$ and $\nu =0$, then\
$ %\[
\delta'=
\left(
\begin{array}{ccc}
 0&0 & \frac{-c_1}{\rho \sigma}\\
 \frac{- \sigma}{\rho }&0 &0\\
  0 &\frac{- a_2\rho}{\sigma}&0
\end{array}
\right)
$. %\]
If $b_3'=1$ then $\sigma =-a_2\rho$, hence
$ %\[
\delta'=
\left(
\begin{array}{ccc}
 0&0 & \frac{c_1}{a_2\rho ^2}\\
a_2&0 &0\\
  0 &1&0
\end{array}
\right)
$. %\]
So, $a_2$ is an invariant in the isomorphism class and $c_1$ may be chosen up to
positive scalar. Hence, the list of isomorphisms classes in case $b_3\neq 0$,
 $a_2\neq 0$ consists of %the following representatives
\[\delta _{a_2,c_1}=
\left(
\begin{array}{ccc}
 0&0 & c_1\\
a_2&0 &0\\
  0 &1&0
\end{array}
\right):a_2>0; c_1=0,\pm 1
\]

%%%%%%%%%%%%%%%%%%%%%%%%%%%%%%%%%
%\subsubsection{Case $a_2=0$}
%%%%%%%%%%%%%%%%%%%%%%%%
\noindent {\em Case $a_2=0$},  $\delta=
\left(
\begin{array}{ccc}
a_1&0&c_1\\
0&0&a_1\\
0&b_3&0
\end{array}
\right)
$. %By means of 
Applying $\phi_{\mu ,\nu,\rho,\sigma}^{a,b}$ with $a=0=b=\rho=\sigma$, $\nu=1$% we obtain
\[
\delta '=
\left(
\begin{array}{ccc}
a_1&0&c_1/\mu\\
0&0&a_1\\
0&b_3/\mu &0
\end{array}
\right)
\]
so we may assume $b_3=1$.
Using a general automorphism we obtain
\[\delta '=\frac{1}{\mu\nu-\rho\sigma}
\left(
\begin{array}{ccc}
      a_1 \mu + b \sigma& b \nu + a_1 \rho& 2 a a_1 + b^2 + c_1\\
      \sigma^2& \nu \sigma& a_1 \mu + b \sigma\\
      \nu \sigma& \nu^2& b \nu + a_1 \rho
\end{array}
\right)
\]      
To preserve $a_2'=0$ we need $\sigma=0$, so
$%\[
\delta '=\frac{1}{\mu\nu}
\left(
\begin{array}{ccc}
      a_1 \mu & b \nu + a_1 \rho& 2 a a_1 + b^2 + c_1\\
0& 0& a_1 \mu \\
0& \nu^2& b \nu + a_1 \rho
\end{array}
\right)
$. %\]      
Requiring also $b'_1=0$, then $b=-\frac{a_1\rho}{\nu}$, 
$
\delta '=\frac{1}{\mu\nu}
\left(
\begin{array}{ccc}
      a_1 \mu & 0& 2 a a_1 + \left(\frac{a_1\rho}{\nu}\right)^2 + c_1\\
0& 0& a_1 \mu \\
0& \nu^2& 0
\end{array}
\right)
$.      
The condition  $b'_3=1$ implies $\mu=\nu$, then
$
\delta '=
\left(
\begin{array}{ccc}
      \frac{a_1}{ \mu} & 0&\frac{ 2 a a_1 + \left(\frac{a_1\rho}{\mu}\right)^2 + c_1}
{\mu^2}\\
0& 0& \frac{a_1 }{\mu }\\
0& 1& 0
\end{array}
\right)
$.

\noindent We distinguish the cases $a_1\neq 0$ and $a_1=0$.
If $a_1\neq 0$, set $\mu=a_1$ to get  $a_1'=1$, then
%can choose it equal to 1 and (with $\mu=1$)
\[\delta '=
\left(
\begin{array}{ccc}
1 & 0& \left(2 a + \rho^2 + c_1\right)\\
0& 0&1 \\
0& 1& 0
\end{array}
\right)
\]      
then $a$ may be chosen such that $c'_1=0$, so the
isoclass has only one representative \newline
$%\[
\delta _0=
\left(
\begin{array}{ccc}
1 & 0&0\\
0& 0&1 \\
0& 1& 0
\end{array}
\right)
$. %\]      
In case $a_1=0$ we have
$%\[
\delta '=
\left(
\begin{array}{ccc}
      0 & 0& c_1/\mu^2\\
0& 0& 0 \\
0& 1& 0
\end{array}
\right)
$, %\]      
so $c_1$ may be chosen up to positive scalar. Hence the set of isoclasses 
with $a_2=0$ and $b_3\neq 0$
 is
\[\left\{
\delta _0=
\left(
\begin{array}{ccc}
1 & 0&0\\
0& 0&1 \\
0& 1& 0
\end{array}
\right);\
\delta _{c_1}=
\left(
\begin{array}{ccc}
0 & 0&c_1\\
0& 0&0 \\
0& 1& 0
\end{array}
\right):c_1=0,\pm 1\right\}
\]      

%%%%%%%%%%%%%%
\noindent{\em Case  $b_3=0$}
%%%%%%%%%%%%%%%%
If we take an automorphism $\phi_{\rho =0}$,
$ %\[
\delta '=
\left(
\begin{array}{ccc}
* & \frac{aa_3+b_1}{\mu}  &*\\
\frac{a_2\mu+2a_3\sigma}{\nu}& a_3&* \\
a_3&0& \frac{aa_3+b_1}{\mu}
\end{array}
\right)
$. %\]     
%We distinguish two subcases, namely $a_3=0$ or $a_3\neq 0$.

%%%%%%%%%%%%%%%%%%%
%\subsubsection{Subcase $b_3=0$ and $a_3\neq 0$}
%%%%%%%%%%%%%%%%%
{\em Subcase $b_3=0$ and $a_3\neq 0$}.
If  $a_3\neq 0$, we may choose  $\sigma$ and $\mu$ such that $a_2'=0$
and $a$ such that $b_1'=0$, so we may start with
\[\delta =
\left(
\begin{array}{ccc}
a_1 & 0 &c_1\\
0& a_3&a_1 \\
a_3&0& 0
\end{array}
\right)
\]     
By means of $\phi_{\mu ,\nu,\rho,\sigma}^{a,b}$ with $\rho=0=\sigma =a$, we get 
$\delta '=
%\[\delta '=
\left(
\begin{array}{ccc}
\frac{a_1+a_3b}{\nu} & 0& \frac{c_1}{\mu\nu}\\
0& a_3&\frac{a_1+a_3b}{\nu} \\
a_3&0& 0
\end{array}
\right)
$.\newline %\]     
Besides, if we take a convenient $b$, we get $a'_1=0$. So we may start with
$%\[
\delta =
\left(
\begin{array}{ccc}
0 & 0&c_1\\
0& a_3&0 \\
a_3&0& 0
\end{array}
\right)
$;      
 applying $\phi_{\mu,\nu,\rho,\sigma}^{a,b}$ to the previous $\delta$, we get
\[\delta '=\frac{a_3}{\mu\nu-\rho\sigma}
\left(
\begin{array}{ccc}
         b \mu + a \sigma& a  \nu +  b \rho& 2 a  b + c_1/a_3\\
          2  \mu \sigma&  \mu \nu + \rho \sigma & b \mu + a \sigma\\
           \mu \nu + \rho \sigma& 2 \nu \rho& a \nu +  b \rho
\end{array}
\right)
\]
In order to preserve the
 conditions $a_1=0=a_2=b_1=b_3$, we need to solve the equations
$%\[
          b \mu + a \sigma=0,\  a \nu + b\rho =0,\ 
           \mu \sigma =0,\ 
          \nu \rho =0
$. %\]
The first two equations imply $a=b=0$ and the last two say that the submatrix 
$
\left(
\begin{array}{cc}
         \mu & \rho\\
          \sigma&  \nu 
\end{array}
\right)
$ has to be one the two possibilities
$%\[
\left(
\begin{array}{cc}
         \mu & 0\\
          0&  \nu 
\end{array}
\right)
\hbox{ or }
\left(
\begin{array}{cc}
         0 & \rho\\
          \sigma&  0
\end{array}
\right)
$. %\]
Hence, after appyling each of these automorphisms, 
the given $\delta =
\left(
\begin{array}{ccc}
0 & 0&c_1\\
0& a_3&0 \\
a_3&0& 0
\end{array}
\right)$
transforms respectively into
\[
\delta '=
\left(
\begin{array}{ccc}
0 & 0&\frac{c_1}{\mu\nu}\\
0& a_3&0 \\
a_3&0& 0
\end{array}
\right)
\hbox{ or }
\delta '=\left(
\begin{array}{ccc}
0 & 0&-\frac{c_1}{\rho\sigma}\\
0& -a_3&0 \\
-a_3&0& 0
\end{array}
\right)
\]
then we may choose $a_3$ up to sign and $c_1$ up to scalar.

We finally conclude that an exhaustive list of isomorphism classes in the case $a_3\neq 0$, $b_3=0$, consists of the Lie bialgebras 
with the following cobrackets
\[\delta  _{a_3,0}=
\left(
\begin{array}{ccc}
0 & 0&0\\
0& a_3&0 \\
a_3&0& 0
\end{array}
\right);\quad
\delta _{a_3,1}=\left(
\begin{array}{ccc}
0 & 0&1\\
0& a_3&0 \\
a_3&0& 0
\end{array}
\right)
:a_3>0
\]
%%%%%%%%%%%%%%%%%%
%\subsubsection{Subcase $b_3=0$ and $a_3= 0$}
%%%%%%%%%%%%%%%%%
{\em Subcase $b_3=0$ and $a_3= 0$}.
We start in this case with $\delta =\left(
\begin{array}{ccc}
a_1 & b_1&c_1\\
a_2& 0& a_1\\
0&0& b_1
\end{array}
\right)$; after applying a general automorphism it transforms into
$\delta '=$\[
\frac{1}{\mu\nu-\rho\sigma}
\left(\begin{array}{ccc}
 a_1 \mu + a a_2 
      \mu + b_1 \sigma & b_1 \nu + a_1 \rho  + a 
          a_2 \rho & _2 a a_1 + a^2 a_2 + _2 b b_1 + c_1\\
        a_2 \mu^2& a_2 \mu \rho & a_1 \mu + a a_2 \mu + b_1 \sigma \\
        a_2 \mu \rho & a_2 \rho ^2& b_1 \nu + a_1 \rho  + a a_2 \rho 
\end{array}
\right)
\]
We distinguish the cases  $a_2\neq 0$ or $a_2=0$.
{\em Suppose $b_3=0$, $a_3=0$ and $a_2\neq 0$}.
The maps preserving $b_3=0$ are those with $\rho=0$, then
\[
\delta '=\frac{1}{\mu\nu}
\left(\begin{array}{ccc}
 a_1 \mu + a a_2 
      \mu + b_1 \sigma & b_1 \nu  & 2 a a_1 + a^2 a_2 + 2 b b_1 + c_1\\
        a_2 \mu^2& 0 & a_1 \mu + a a_2 \mu + b_1 \sigma \\
        0 &0& b_1 \nu  
\end{array}
\right)
\]
Besides, if we take $\sigma=0$, $a=-\frac{a_1}{a_2}$, we get $a_1'=0$; so we may start with 
$\delta =\left(
\begin{array}{ccc}
0 & b_1&c_1\\
a_2& 0& 0\\
0&0& b_1
\end{array}
\right)
$.
Applying an automorphism with $\rho=0$, it maps to
\[
\delta '=\frac{1}{\mu\nu}
\left(\begin{array}{ccc}
 aa _2      \mu + b_1 \sigma & b_1 \nu  &  a^2 a_2 + 2 b b_1 + c_1\\
        a_2 \mu^2& 0 &  a a_2 \mu + b_1 \sigma \\
        0 &0& b_1 \nu  
\end{array}
\right)
\]
The condition $a_1'=0$ preserves if  $a=\frac{- b_1 \sigma }{a_2\mu}$, so 
$%\[
\delta '=
\left(\begin{array}{ccc}
 0& \frac{b_1}{\mu}  &  \frac{ 2 b b_1 + c_1+\frac{b_1^2 \sigma^2 }{a_2\mu^2}}{\mu\nu}\\
       \frac{ a_2 \mu}{\nu}& 0 &  0\\
        0 &0&\frac{ b_1}{\mu  }
\end{array}
\right)
$; %\]
then $a_2$ can be chosen up to scalar. If we choose $a_2=1$ and also $\mu=\nu$,
then
$%\[
\delta '=
\left(\begin{array}{ccc}
 0& b_1 /\mu  &  \frac{ 2 b b_1 + c_1+\frac{b_1^2 \sigma^2 }{\mu^2}}{\mu^2}\\
       1& 0 &  0\\
        0 &0& b_1/\mu  
\end{array}
\right)
$. %\]
Hence  $b_1$ is determined up to scalar, so the possibilities are $b_1=0$ or $b_1=1$.

\noindent {\em Subcase $b_1=0$}:
$%\[
\delta '=
\left(\begin{array}{ccc}
 0& 0  &  \frac{ c_1}{\mu^2}\\
       1& 0 &  0\\
        0 &0& 0 
\end{array}
\right)
$, %\]
then $c_1$ may be chosen up to positive scalar.
The list of isomorphisms classes in this case is given by 
$%\[
\left\{
\delta =
\left(\begin{array}{ccc}
 0& 0  &  c_1\\
       1& 0 &  0\\
        0 &0& 0 
\end{array}
\right):c_1=0,\ \pm1\right\}
$. %\]
{\em Subcase $b_1=1$}. Let $\mu=1$ then
$%\[
\delta '=
\left(\begin{array}{ccc}
 0&1 &  2b+c_1+\sigma ^2\\
1&0&0\\
        0 &0&1  
\end{array}
\right)
$. %\]
We may choose $c_1=0$. In this case there is a unique isoclass given by
%Let $c_1=0$: so the unique isoclass is
$%\[
\delta =
\left(\begin{array}{ccc}
 0&1 &  0\\
1&0&0\\
        0 &0&1  
\end{array}
\right)
$.\\ %\]
\noindent {\em Suppose $b_3=0$, $a_3=0$ and $a_2=0$}.
We start with $
\delta =
\left(\begin{array}{ccc}
 a_1&b_1 & c_1\\
0&0&a_1\\
        0 &0&b_1  
\end{array}
\right)$.
After applying a general isomorphism it is mapped into
\[
\delta ' =\frac{1}{\mu\nu-\rho\sigma}
\left(\begin{array}{ccc}
 a_1\mu +b_1\sigma &a_1\rho+b_1\nu  &2aa_1+2bb_1+ c_1\\
0&0&a_1\mu +b_1\sigma \\
        0 &0&a_1\rho+b_1\nu 
\end{array}
\right)
\]
If the pair $(a_1,b_1)\neq (0,0)$ then there exists a linear transformation
$\left(\begin{array}{cc}
 \mu &\sigma \\
\rho&\nu \end{array}
\right) $ of determinant $1$, such that
$(a_1\mu +b_1\sigma ,a_1\rho+b_1\nu )=(1,0)$. 
This says that the cobracket $\delta$ belongs to same isoclass that one with 
$a_1=1$ and $b_1=0$.
If we make such a choice, namely  
$
\delta =
\left(\begin{array}{ccc}
 1&0 & c_1\\
0&0&1\\
        0 &0&0
\end{array}
\right)$, then it transforms under a general automorphism into
\[
\delta '=\frac{1}{\mu\nu-\rho\sigma}
\left(\begin{array}{ccc}
 \mu&\rho & 2a+c_1\\
0&0&\mu\\
        0 &0&\rho
\end{array}
\right)\]
so, if we whish to preserve $b'_1=0$, it must be $\rho =0$, then $\delta '$ equals
$%\[
\delta '=
\left(\begin{array}{ccc}
\frac{1}{ \nu}&0 &\frac{2a+c_1}{\mu\nu}\\
0&0&\frac{1}{ \nu}\\
        0 &0&0
\end{array}
\right)$. %\]
In order to preserve also $a'_1=1$ we need $\nu=1$; then
$\delta '=
\left(\begin{array}{ccc}
1&0 &\frac{2a+c_1}{\mu}\\
0&0&1\\
        0 &0&0
\end{array}
\right)$
and we may choose $c_1=0$. Hence, in this case we get only one representative given by
$
\delta =
\left(\begin{array}{ccc}
1&0 &0\\
0&0&1\\
        0 &0&0
\end{array}
\right)
$.

{\em If the pair $(a_1,b_1)= (0,0)$} then 
\[
\delta =
\left(\begin{array}{ccc}
0&0 &c_1\\
0&0&0\\
        0 &0&0
\end{array}
\right)
\hbox{ and }\delta '=
\left(\begin{array}{ccc}
0&0 &\frac{c_1}{\mu\nu-\rho\sigma}\\
0&0&0\\
        0 &0&0
\end{array}
\right)
\]
We conclude that there are two representatives for this case, namely
\[
\delta =
\left(\begin{array}{ccc}
0&0 &c_1\\
0&0&0\\
        0 &0&0
\end{array}
\right):c_1=0\hbox{ or }c_1=1.
\]
%%%%%%%%%%%%%%%%%%%%%%%%%%
%%%%%%%%%%%%%%%%%%%%%%%%%%%%%

%%%%%%%%%%%%%%%%%%%%%%%%%%%%
\section{Lie bialgebra structures on $\r'_{3,\lambda}$}
%%%%%%%%%%%%%%%%%%%%%%
%%%%%%DESDE ACA INSERTE
%%%%%%%%%%%%%%%%%%%%%

In basis $\{h,x,y\}$, the Lie algebra $\r'_{3,\lambda}$ has the following brackets:
$
[h,x]=\lambda x-y$, $ [h,y]=x+\lambda y$, $ [x,y]=0
$.
Remark that, in the basis $\{x,y\}$, the linear transformation $\ad_h$ has matrix
$\left(\begin{array}{cc}
\lambda&-1\\
1&\lambda
\end{array}\right)$. Notice the similarity with the matrix associated to multiplication
by the number $\lambda-i$.
A straightforward computation shows
$(\Lambda ^2\g)^\g=0$ if $\lambda\neq 0$
and $(\Lambda ^2\g)^\g=\R x\w y$ if $\lambda =0$,  for $\g=\r'_{3,\lambda}$.

\

\noindent {\bf 1-cocycle condition}. Sea $\delta :\g\to\Lambda ^2\g$ a 1-cocycle.
From $0=\delta[x,y]=[\delta x,y]+[x,\delta y]$ we get
\begin{eqnarray*}
0&=&[a_1x\w y+a_2 y\w h+a_3 h\w x,y]+[x,b_1x\w y+b_2 y\w h+b_3 h\w x]\\ %\]
%\[=[a_2 y\w h+a_3 h\w x,y]+[x,b_2 y\w h+b_3 h\w x]\]
%\[=a_2 y\w [h,y ]+a_3 [h,y]\w x+b_2 y\w [x,h]+b_3 [x,h]\w x\]
%\[=a_2 y\w x+a_3 \lambda y\w x+b_2 y\w (-\lambda x)+b_3 y\w x\]
&=&(a_2+a_3 \lambda -b_2 \lambda +b_3 )y\w x
\end{eqnarray*}
hence, {$a_2+a_3\lambda-\lambda b_2+b_3=0$}. The 1-cocycle condition for $[h,x]=\lambda x -y$ gives
\begin{eqnarray*}
\lambda\delta x-\delta y&=&
[\delta h,x]+[h,\delta x]\\
&=&[c_1x\w y+c_2 y\w h+c_3 h\w x,x]+[h,a_1x\w y+a_2 y\w h+a_3 h\w x]\\
%\[=\lambda c_2 y\w x-c_3 y\w x+[h,a_1x\w y+a_2 y\w h+a_3 h\w x]\]
%\[=(\lambda c_2 -c_3) y\w x+a_1[h,x]\w y+a_1x\w [h,y]+a_2[h, y]\w h+a_3 h\w[h, x]\]
%\[=(\lambda c_2 -c_3) y\w x+a_1\lambda x\w y+\lambda a_1x\w y+a_2(x+\lambda y)\w h+a_3 h\w(\lambda x-y)\]\[
&=&(-\lambda c_2 +c_3+2\lambda a_1) x\w y+(a_2\lambda +a_3)y\w h+(-a_2+\lambda a_3) h\w x)
\end{eqnarray*}
So
\begin{eqnarray*}
\lambda a_1-b_1&=&-\lambda c_2 +c_3+2\lambda a_1\\
\lambda a_2-b_2&=&a_2\lambda +a_3\\
\lambda a_3-b_3&=&-a_2+\lambda a_3
\end{eqnarray*}
then %equivalently
$\lambda a_1+b_1=\lambda c_2 - c_3$, $-b_2=a_3$, $b_3=a_2$.
%\[\fbox{$-b_2=a_3$}\]
%\[\fbox{$b_3=a_2$}\]
Similarly, %the cocycle condition for 
$[h,y]=x+\lambda y$ gives
\begin{eqnarray*}
 \delta x+\lambda \delta y&=&[\delta h,y]+[h,\delta y]\\
&=&
(c_2+c_3 \lambda )y\w x+2\lambda b_1 x\w y+(b_2\lambda +b_3)y\w h+(-b_2+\lambda b_3) h\w x
\end{eqnarray*}
then
\[a_1+\lambda b_1=-c_2-\lambda c_3+2\lambda b_1,\quad 
a_2+\lambda b_2=\lambda b_2+b_3,\quad
a_3+\lambda b_3=-b_2+\lambda b_3\]
Summarizing, we get
\[a_2=b_3,\quad a_3=-b_2,\]
\[\lambda b_2-b_3=\lambda a_3+a_2,\quad 
\lambda b_1-a_1=c_2+\lambda c_3,\quad 
b_1+\lambda a_1=\lambda c_2 - c_3\]
The last two equations are equivalent to
\[
c_2=\frac{a_1(\lambda^2-1)+2\lambda b_1}{1+\lambda^2},\quad 
c_3=\frac{b_1(\lambda^2-1)-2\lambda a_1}{1+\lambda^2}
\]
while the first three ones are equivalent to
\[a_2=-\lambda a_3,\ \
b_2=-a_3,\ \ 
b_3=-\lambda a_3\]
The  general 1-cocycle, in basis $\{x,y,h\}$, $\{x\w y, y\w h, h\w x\}$
is given by
\[
\left(
\begin{array}{ccc}
a_1&b_1&c_1\\
-\lambda a_3&-a_3&\frac{a_1(\lambda^2-1)+2\lambda b_1}{1+\lambda^2}\\
a_3&-\lambda a_3&\frac{b_1(\lambda^2-1)-2\lambda a_1}{1+\lambda^2}
\end{array}\right)
\]
For a 1-cocycle, the co-Jacobi condition is
$%\[
2 \frac{ (a_1^2 + b_1^2)\lambda + a_3 c_1(1 + \lambda ^2)}{1 + \lambda ^2}=0
$. %\]

%%%%%%%%%%%%%%%%%%%%%%%%
%{\bf Automorphism group}.
%%%%%%%%%%%%%%%%%%%%%%%%%%

\begin{proposition}
The automorphism group of the Lie algebra $\r'_{3,\lambda}$, expresed as matrices
in basis $\{x,y,h\}$,  is the following  subgroup of $\GL(3, \R)$
\[
\left\{
\left(\begin{array}{ccc}
\mu&-\sigma&a\\
\sigma&\mu&b\\
0&0&1
\end{array}
\right): \mu,\sigma,a,b\in\R, \ \mu^2+\sigma^2\neq 0
\right\}
\]
\end{proposition}

\begin{proof} %Let us consider the ordered basis $\{x,y,h\}$. 
Using that $[\g,\g]$ is generated by $x$ and $y$ and it is invariant under automorphism, we conclude
that any automorphism
$\phi$ restricted to $[\g,\g]$ must be of the form
$\phi(x)=\mu x+\rho y$ and $\phi y=\sigma x+\nu y$, with $\mu\nu-\rho\sigma\neq 0$.
Also, writing $\phi(h)=ax+by+ch$, since $\phi$ is an automorphism of the Lie algebra, we have
\[c[h,\phi x]=
[\phi h,\phi x]=\lambda\phi x-\phi y,\quad
c[h,\phi y]=[\phi h,\phi y]=\phi x+\lambda \phi y,
\]
in matrix notation,
\[c
\left(\begin{array}{cc}
\mu&\sigma\\
\rho&\nu
\end{array}
\right)
\left(\begin{array}{cc}
\lambda&-1\\
1&\lambda
\end{array}
\right)
=
\left(\begin{array}{cc}
\lambda&-1\\
1&\lambda
\end{array}
\right)
\left(\begin{array}{cc}
\mu&\sigma\\
\rho&\nu
\end{array}
\right)
\]
Taking determinant we get $c=1$, and if
$\left(\begin{array}{cc}
\mu&\sigma\\
\rho&\nu
\end{array}
\right)$ commutes with the matrix 
$\left(\begin{array}{cc}
0&-1\\
1&0
\end{array}
\right)$ then it must be of the form 
$\left(\begin{array}{cc}
\mu&-\sigma\\
\sigma&\mu
\end{array}
\right)$.
\end{proof}
{\bf Action of the automorphisms group on 1-cocycles}.
The efect of an arbitrary automorphism on a general 1-cocycle 
 is the following $\delta '=$
\[=\left(\begin{array}{ccc}
\frac{a_1 \mu  + b_1 \sigma - a_3 (\mu ( a \lambda -b)  +  \sigma (a+ b \lambda ))}{\mu  ^2+\sigma ^2}&
\frac{b_1 \mu -a_1 \sigma   - a_3 (b (\lambda  \mu   + \sigma) + a  (\mu  - \lambda  \sigma))}{\mu ^2+\sigma ^2}&
c_1'\\
-a_3\lambda & -a_3 &*\\
a_3 & -a_3\lambda &*
\end{array}
\right)\]
with
$$c_1'=\frac{c_1 (1 + \lambda ^2) -   \lambda  (-2 a (b_1 + a_1 \lambda ) + a^2 a_3 (1 + \lambda ^2) +b (2 a_1 - 2 b_1 \lambda  + a_3 b (1 + \lambda ^2)))}
{(\mu ^2+\sigma ^2)(1+\lambda ^2)}
$$
Recall that co-Jacobi condition reads 
$%\[
2 \frac{ (a_1^2 + b_1^2)\lambda + a_3 c_1(1 + \lambda ^2)}{1 + \lambda ^2}=0
$, %\]
or, equivalently
$%\[
 (a_1^2 + b_1^2)\lambda + a_3 c_1(1 + \lambda ^2)=0$. %\]
We will make some simplifications using the automorphism group.

\noindent{\em Case $a_3\neq 0$}.
If we choose an automorphism with $\mu=1$, $\sigma=0$, then $\delta'$ has $a_1'=a_1+a_3(b-\lambda a)$ 
and $b_1'=b_1-(\lambda b+a)$, so we can choose
$a$ and $b$ such that $a_1'=0=b'_1$; hence, we may suppose from the begining that $a_1=0$
and $b_1=0$. But now $a_1=b_1=0$ together with co-Jacobi imply $a_3 c_1(1 + \lambda ^2)=0$
so $a_3c_1=0$; since $a_3\neq 0$,  $c_1=0$, hence
\[
\delta_{a_3,\lambda }=
\left(
\begin{array}{ccc}
0&0&0\\
-\lambda a_3&-a_3&0\\
a_3&-\lambda a_3&0
\end{array}\right)
\]
To compute the automorphism group, notice that  $\phi_{\mu,\nu,a,b}$
transforms $\delta_{a_1=b_1=c_1=0}$ into $\delta'$ with $c_1'=-\lambda a_3\frac{a^ 2+b^ 2}{
\mu^ 2+\nu^ 2}$,
so in case $\lambda\neq 0$ the only possibility to preserve $c_1=0$ is $a=b=0$.
Hence,  the automorphism group in case $a_3\neq 0$, $\lambda \neq 0$ is
\[
\left\{
\left(\begin{array}{ccc}
\mu&-\sigma&0\\
\sigma&\mu&0\\
0&0&1
\end{array}
\right): \mu,\sigma\in\R, \ \mu^2+\sigma^2\neq 0
\right\}
\]
But if  $\lambda =0$,  $\delta_{a_1=b_1=c_1=0}$ transforms by  $\phi_{\mu,\nu,a,b}$
into 
\[\delta '=
\left(
\begin{array}{ccc}
a_3\frac{b\mu -a\sigma}{\mu ^2+\sigma ^2} 
& -a_3\frac{b\sigma +a\mu}{\mu ^2+\sigma ^2} &  0\\
0 & -a_3&  - a_3\frac{b\mu -a\sigma}{\mu ^2+\sigma ^2}\\
a_3  &0 & a_3\frac{b\sigma +a\mu}{\mu ^2+\sigma ^2}
\end{array}\right)
\]
Since $\left(
\begin{array}{cc}
\mu & -\sigma\\
\sigma & \mu 
\end{array}\right)$ is invertible, the only way to preserve $a_1=b_1=0$
is with $a=b=0$. Hence,  the automorphism group in case $a_3\neq 0$, $\lambda = 0$ is
the same as in case $\lambda \neq 0$.

\noindent {\em Case $a_3=0$}. If $\lambda \neq 0$, co-Jacobi implies  $a_1=b_1=0$;
conjugation by   $\phi_{\mu,\nu,a,b}$ gives
\[\delta =
\left(
\begin{array}{ccc}
0& 0 &  c_1\\
0 & 0&  0\\
0  &0 & 0
\end{array}\right)
\mapsto
\delta '=
\left(
\begin{array}{ccc}
0& 0 &  \frac{c_1}{\mu ^2+\sigma ^2}\\
0 & 0&  0\\
0  &0 & 0
\end{array}\right)
\]
so $c_1$ can be choosen up to positive scalar. We may take $0,\pm 1$ as representatives.
In case $a_3=0$ but $\lambda = 0$, co-Jacobi identity gives no further information. We study the action of the automorphisms group in this case.
$\delta_{a_3=0,\lambda = 0}$ transforms by  $\phi_{\mu,\nu,a,b}$
into 
\[\delta '=
\left(
\begin{array}{ccc}
\frac{a_1\mu +b_1\sigma}{\mu ^2+\sigma ^2} 
& \frac{-a_1\sigma +b_1\mu}{\mu ^2+\sigma ^2}
 & \frac{c_1}{\mu ^2+\sigma ^2}  \\
0 & 0  & -\frac{a_1\mu +b_1\sigma}{\mu ^2+\sigma ^2}  \\
0  &0 & \frac{a_1\sigma -b_1\mu}{\mu ^2+\sigma ^2}
\end{array}\right)
\]
The pair $(a_1,b_1)$ transforms as
 $a_1+ib_1\mapsto \frac{a_1+ib_1}{\mu +i\sigma}$ in the complex plane.
We know that there are two orbits: $(a_1,b_1)=(0,0)$ wich has trivial action
and it gives the same cobrackets as for $\lambda \neq 0$, and
$\{(a_1,b_1)\neq (0,0)\}$, which has free $\C ^*$-action. For the second case,
one can take $(a_1,b_1)=(1,0)$ as a representative. 
We conclude that a set of representatives of Lie cobrackets in case $a_3=0$, $\lambda = 0$ is given by
\[\delta =
\left(
\begin{array}{ccc}
0& 0 &  c_1\\
0 & 0&  0\\
0  &0 & 0
\end{array}\right):c_1=0,\pm 1,\
\delta_{c_1}^{(1,0)} =
\left(
\begin{array}{ccc}
1& 0 &  c_1\\
0 & 0&  -1\\
0  &0 & 0
\end{array}\right):c_1\in\R
\]
The automorphisms group of the Lie bialgebra with $\delta_{c_1}^{(1,0)} $
is
\[
\left\{
\left(\begin{array}{ccc}
1& 0&a\\
0 & 1 &b\\
0&0&1
\end{array}
\right): a,b\in\R .
\right\}
\]
\begin{theorem}
The set of isomorphism classes of Lie bialgebra with underlying Lie algebra $\r'_{3,\lambda}$ in the case
$\lambda\neq 0$ is given by the following list of cobrackets:
\[
\delta_{a_3,\lambda }=\left(
\begin{array}{ccc}
0&0&0\\
-\lambda a_3&-a_3&0\\
a_3&-\lambda a_3&0
\end{array}\right): a_3\in\R\hbox{ and }
\delta_{0,c_1}=\left(
\begin{array}{ccc}
0&0&\pm 1\\
0&0&0\\
0&0&0
\end{array}\right).
\]
In case $\lambda =0$, we have the previous set specialized in  $\lambda =0$, 
together with the following 1-parameter familly
$\delta_{c_1}^{(1,0)} =
\left(
\begin{array}{ccc}
1& 0 &  c_1\\
0 & 0&  -1\\
0  &0 & 0
\end{array}\right):c_1\in\R
$.
\end{theorem}

%%%%%%%%%%%%%
%%%%%%%%%%%%%%%%%%%%%%%%%%%%
%%%%%%%%%%%%%HASTA ACA INSERTE $\r'_{3,\lambda}$
%%%%%%%%%%%%%%%%%%%%%%%%%%%%%%%%

%%%%%%%%%%%%%%%%%%%%%%%%%%
\section{Lie bialgebra structures on $\su(2)$}
%%%%%%%%%%%%%%%%%%%%%%%%%%

{\bf 1-Cocycles}.
Consider $\su(2)$ as the $\R$-span of the following matrices:
\[u=\frac{1}{2}\left(\begin{array}{cc}
i&0\\
0&-i
\end{array}\right),\
v=\frac{1}{2}\left(\begin{array}{cc}
0&i\\
i&0
\end{array}\right),\
w=\frac{1}{2}\left(\begin{array}{cc}
0&-1\\
1&0
\end{array}\right)\]
then the Lie brackets verify
$[u,v]=w$, $[v,w]=u$, $[w,u]=v$. This is a simple Lie algebra, then every 1-cocycle is a 1-coboundary.
If $r=\alpha u\w v+\beta v\w w+\gamma w\w u\in\Lambda^2\su(2)$, 
with $\alpha$, $\beta$, $\gamma\in\R$, 
the 1-cocycle associated to it is
$\delta(x)=\ad_x(r)=[x,r]$ for any $x\in\su(2)$. 
The Co-Jacobi condition 
for $\delta$ is equivalent to $[r,r]\in (\Lambda ^3\g)^\g$ with
$[r,r]=2(\alpha ^2+\beta ^2+\gamma ^2)u\w w\w v$, so it is satisfied for any $r$ since  $(\Lambda ^3\g)^\g=\Lambda ^3\g$ for $\g=\su(2)$.
We get
\begin{eqnarray*}\delta(u)=\gamma u\w v-\alpha w\w u;\ %\]\[
\delta(v)=-\beta u\w v+\alpha v \w w;\ %\]\[
\delta(w)=-\gamma v\w w+\beta w \w u;\ %\]
\end{eqnarray*}
or, in matrix notation
$%\[
\delta =
\left(
\begin{array}{ccc}
\gamma &-\beta &0\\
0&\alpha &-\gamma \\
-\alpha &0&\beta 
\end{array}
\right).
$%\]

{\bf Automorphisms and isomorphism classes}.
%Consider $\su(2)$ as the $\R$-span of the following matrices:
%\[u=\frac{1}{2}\left(\begin{array}{cc}
%i&0\\
%0&-i
%\end{array}\right),\
%v=\frac{1}{2}\left(\begin{array}{cc}
%0&i\\
%i&0
%\end{array}\right),\
%w=\frac{1}{2}\left(\begin{array}{cc}
%0&1\\
%-1&0
%\end{array}\right)\]
If $U\in \SU(2)$, then conjugation by $U$ gives an automorphism of the
Lie algebra $\su(2)$.
If we parametrize such a matrix by
%\[U=\left(\begin{array}{cc}
%A+iB&C+iD\\
%-C+iD&A-iB
%\end{array}\right)\]
%with $A,B,C,D\in\R$,   $A^2+B^2+C^2+D^2=1$,  we have the automorphism
%$\phi_{ABCD}(M)=UMU^{-1}$,
\[
U=
\left(\begin{array}{cc}
a+ib&c+id\\
-c+id&a-ib
\end{array}\right)
\]
with $a,b,c,d\in\R$,   $a^2+b^2+c^2+d^2=1$,  we have the automorphism 
$\phi_{U}(M)=UMU^{-1}$,
where $M\in\su(2)$.
Straightforward computation shows the following:
\[
\begin{array}{rcccc}
\phi_{U}(u)&=&(a^2+b^2-c^2-d^2)u& +2(-ac+bd)v&-2 (bc+ad)w\\
\phi_{U}(v)&=&2(ac+bd)u& +(a^2-b^2-c^2+d^2)v&+2 (ab-cd)w\\
\phi_{U}(w)&=&-2(bc-ad)u&-2(ab+cd)v&+(a^2-b^2+c^2-d^2)w
\end{array}
\]
and $\delta':=(\phi\w\phi)^{-1}\delta\phi$ is given by $\frac{\delta'}{2}=$
%\[2\left(\begin{array}{c|c|c}
%a(A B-CD)-b(AC+BD)&-a(BC+AD)+c( BD-AC)&   0  \\
%+\frac12 c(A^2-B^2-C^2+D^2) &+\frac12 b (-A^2-B^2+C^2+D^2) &         \\
%\hline
%0                                      &\frac12 a(A^2-B^2+C^2-D^2)   & a(CD- A B) %+b(BD+AC)\\
%                                       &+b(BC-AD)-c(AB+CD) &+\frac12 c(- A^2  +   B^2   +  C^2 %-  D^2)\\
%\hline
%\frac12 a(-A^2+B^2-C^2+ D^2) & 0& a(AD+BC)+c(AC-BD)\\
%+(BC+AD) +(A B  +  C D)&&+\frac12 b(A^2+B^2-C^2-D^2)
%\end{array}\right)
%\]
\[
\left(\begin{array}{c|c|c}
\alpha (a b-cd)+\beta (ac+bd)&\alpha (ad+bc)+\gamma (ac-bd)&   0  \\
+\frac12 \gamma (a^2-b^2-c^2+d^2) &+\frac12 \beta  (-a^2-b^2+c^2+d^2) &         \\
\hline 0                                 &\frac12 \alpha (a^2-b^2+c^2-d^2)   & \alpha (cd- ab) -\beta (bd+ac)\\
                                       &\beta (ad-bc)-\gamma (ab+cd) &+\frac12 \gamma (- a^2  +b^2 +c^2 -d^2)\\
\hline
\frac12 \alpha (-a^2+b^2-c^2+ d^2) & 0&- \alpha (ad+bc)+\gamma(bd-ac)\\
+\beta (bc-ad) +\gamma (ab  + cd)&&+\frac12 \beta (a^2+b^2-c^2-d^2)
\end{array}\right)
\]

Suppose $\gamma \neq 0$, and take $b=d=0$,
then
\[
\delta'=\left(\begin{array}{ccc}
\gamma (a^2-c^2)+2\beta ac&2\gamma ac+\beta  (c^2-a^2) &        0 \\
0                              &\alpha (a^2+c^2)       & -2\beta ac+\gamma (c^2- a^2)\\
-\alpha  (a^2+c^2) & 0& -2\gamma ac+\beta (a^2-c^2)
\end{array}\right)
\]
So, if %$C=A\left(\dfrac{-b +\sqrt{b^2+c^2}}{c}\right)$ then $c'=0$. 
$c=a\left(\frac{\beta +\sqrt{\beta ^2+\gamma ^2}}{\gamma }\right)$ then 
$\gamma '=0$. 
Assuming $\gamma =0$ and letting $a=d=0$, we get %$A=D=0$, we get \delta'$:
\[
\delta'=
\left(\begin{array}{ccc}
0&2\alpha bc+\beta  (-b^2+c^2)&   0  \\
0&-\alpha (b^2-c^2) -2\beta bc     &0\\
\alpha  (b^2-c^2) +2\beta bc& 0& - 2\alpha bc+\beta (b^2-c^2)
\end{array}\right)
\]
So if  $\beta \neq 0$,  we can use  $\phi_{U}$  with $a=d=0$, 
$c=b\left(\frac{-\alpha +\sqrt{\alpha ^2+\beta ^2}}{\beta }\right)$ then $\delta'$ has $\gamma '=0$ and $\beta '=0$. So we may assume from the beginning
$\gamma =\beta =0$.
We get 
\[
\delta'=\alpha
\left(\begin{array}{ccc}
2 (ab-cd)&2 (ad+bc)&   0  \\
0                                      & (a^2-b^2+c^2-d^2)  & 2  (cd-ab)\\
 (-a^2+b^2-c^2+d^2)& 0& -2(ad+bc)\\
\end{array}\right)
\]
In case $\alpha =0$, we have the trivial cobracket. If $\alpha \neq 0$ , the condition on the automorphism preserving
the condition $\beta =\gamma =0$ is given by the equations:
\[ab=cd,\ ad=-bc\]
These equations imply $a^2b=-bc^2$, so if $b\neq 0$, then $a=c=0$ and $\delta'$ is given by
\[
\delta'=
\left(\begin{array}{ccc}
0&0&   0  \\
0  &-\alpha (b^2+d^2)  & 0\\
\alpha ( b^2 +  d^2)& 0&0\\
\end{array}\right)
\]
but $b^2+d^2=1$, so $\delta_\alpha\cong \delta_{-\alpha}$.

\noindent Next we consider the case $b=0$. We have $cd=0=ad$.
If $d=0$, %we have 
\[\delta'=
\left(\begin{array}{ccc}
0&0&   0  \\
0  &\alpha (a^2+c^2)  & 0\\
-\alpha (a^2 + c^2)& 0&0\\
\end{array}\right)=\delta
\]
If $d\neq 0$ then $a=c=0$, $d=\pm 1$ and
$%\[
\delta'=
\left(\begin{array}{ccc}
0&0&   0  \\
0  &-\alpha d^2 & 0\\
\alpha d^2& 0&0\\
\end{array}\right)=-\delta
$. %\]

\begin{proposition}
The set of isomorphism classes of Lie bialgebras with underlying
Lie algebra $\su(2)$ is given by the 1-parameter family
\[
\delta_\alpha=
 \left(\begin{array}{ccc}
0&0&   0  \\
0  &\alpha & 0\\
-\alpha & 0&0\\
\end{array}\right): \alpha\geq 0
\]
For each $\alpha\neq 0$, the automorphism group is 
%$\left\{\phi_{U}:U=\left(\begin{array}{ccc}
%a&c\\
%-c&a
%\end{array}\right): a,c\in \R,\ a^2+c^2=1
%\right\}=$
\[
\left\{
\phi_{U}=
\left(\begin{array}{cccc}
a^2-c^2&-2ac&0\\
2ac&a^2-c^2&0\\
0&0&1
\end{array}\right): a,c\in \R,\ a^2+c^2=1
\right\}\cong S^1
\]
\end{proposition}

%\noindent{\em Case  $\su(2)$ revised}
\begin{remark}
%$\su(2)$ is semisimple, so every Lie  bialgebra structure is coboundary.
%If we write $\delta=\ad_{(-)}(r)$ with $r\in\Lambda^2\su(2)$, the co-Jacobi condition
%for $\delta$, in terms of $r$ is simply $[r,r]\in(\Lambda^3\su(2))^{\su(2)}$. This condition %is always satisfied because $(\Lambda^3\su(2))^{\su(2)}=\Lambda^3\su(2)$. Writting
For any $r=\alpha u\w v+\beta v\w w+\gamma w\w u\in \Lambda^2\su(2)$, 
$%\[
[r,r]=2(\alpha ^2+\beta ^2+\gamma ^2)u\w w\w v\in(\Lambda^3\su(2))^{\su(2)}
$. %\]
We have two possibilities:
\begin{itemize}
\item $[r,r]=0$ if and only if $r=0$  if and only if   $\delta=0$.
Notice that the the only triangular structure is the trivial one.
\item $0\neq [r,r]$, and so $(\su(2),\delta)$ with $\delta (-)=\ad_{(-)}(r)$ is almost factorizable.
\end{itemize}
%In the first case, $\delta=0$, and in the second, 
It was known (see \cite{AJ})
that $\su(2)$ admites a unique almost factorizable structure up to scalar multiple. 
This is in perfect agreement with the results of this section.
\end{remark}
%%%%%%%%%%%%%%%%%%%%%%%%%%
\section{Lie bialgebra structures on $\sl(2,\R)$\label{seccionsl2}}
%%%%%%%%%%%%%%%%%%%%%%%%%%

{\bf 1-Cocycles in $\sl(2,\R)$}.
The Lie algebra $\sl(2,\R)$ is usually presented as generated by $\{x,y,h\}$ with brackets
$[h,x]=2x$, $[h,y]=-2y$, $[x,y]=h$. 
But it is convenient to consider instead the ordered basis
$\{u,v, w\}$ of $\sl(2,\R)$  and $\{u\w v,v\w w,w\w u\}$ of $\Lambda^2\sl(2,\R)$, 
where $u=h/2$,  $v=(x+y)/2$ and $w=(x-y)/2$ {\em i.e.}
\[
u=\frac12
\left(\begin{array}{cc}
1&0\\
0&-1
\end{array}\right),\
v=\frac12
\left(\begin{array}{cc}
0&1\\
1&0
\end{array}\right),\
w=\frac12
\left(\begin{array}{cc}
0&1\\
-1&0
\end{array}\right)
\]
In this basis, the brackets are given by
%\[
$[u,v]=w,\
[v,w]=-u,\
[w,u]=-v
$. %\] 
As in the case $\su(2)$, the Lie algebra
$\sl(2,\R)$ is simple, then every 1-cocycle is a 1-coboundary
and the general considerations made for that case hold here.
If $r=\alpha u\w v+\beta v\w w+\gamma w\w u$, then the 1-cocycle associated to it is
$\delta(z)=\ad_z(r)=[z,r]$ for any $z\in\sl(2,\R)$. Hence
\[\delta(u)=-\gamma u\w v-\alpha w\w u,\
\delta(v)=\beta u\w v+\alpha v\w w,\
\delta(w)=\gamma v\w w-\beta w\w u,\]
in matrix notation, 
$%\[
\delta =\left(
\begin{array}{ccc}
-\gamma  & \beta &0\\
0  & \alpha &\gamma \\
-\alpha &0&-\beta 
\end{array}
\right)
$.
Co-Jacobi is automatically satisfied.

\noindent {\bf Automorphisms}.
In a similar way to the $\su(2)$ case, % let  us give the realization of $\sl(2)$ by the $\R$ %span of the following matrices:
%\[u=1/2\left(\begin{array}{cc}
%1&0\\
%0&-1
%\end{array}\right),\
%v=1/2 \left(\begin{array}{cc}
%0&1\\
%1&0
%\end{array}\right),\
%w=1/2 \left(\begin{array}{cc}
%0&1\\
%-1&0
%\end{array}\right)
%\]
if $S\in \SL(2,\R)$, then conjugation by $S$ gives the automorphisms of the
Lie algebra $\sl(2,\R)$.
If $S=
\left(\begin{array}{cc}
a&b\\
c&d
\end{array}\right)$, 
with $a,b,c,d\in\R$,  $ad-bc=1$, then the automorphism
$\phi_{S}$ given by  $\phi_{S}(M)=SMS^{-1}$ for any $M\in\sl(2)$, maps
$u,v,w$ to
%\[\phi_{a,b,c,d}(M)=UMU^{-1} \]
%where $M\in\sl(2)$.
%Straightforward computation shows the following:
\[
\begin{array}{rcccccc}
\phi_{S}(u)&=&(ad+bc)u&+&(cd-ab)v&-&(cd+ab)w\\
\phi_{S}(v)&=&(bd-ac)u&+&\frac{a^2-b^2-c^2+d^2}{2}v&+&\frac{
a^2-b^2+c^2-d^2}{2}w\\
\phi_{S}(w)&=&-(bd+ac)u&+&\frac{a^2+b^2-c^2-d^2}{2}v&+&\frac{a^2+b^2+c^2+d^2}{2}w\\
\end{array}
\]
In matrix notation, 
\[
\phi_{S}=\left(
\begin{array}{ccc}
ad+bc    &         bd-ac                             &       -(bd+ac)\\
cd-ab     &\frac{a^2-b^2-c^2+d^2}{2}&\frac{a^2+b^2-c^2-d^2}{2}\\
-(cd+ab)&\frac{a^2-b^2+c^2-d^2}{2}&\frac{a^2+b^2+c^2+d^2}{2}
\end{array}\right)
\]
Let us denote by 
$\kappa =a^2+b^2+c^2+d^2$, 
$\kappa _{1,3}=-a^2+b^2-c^2+d^2$,
 $\kappa _{3,4}=a^2+b^2-c^2-d^2$, $\kappa _{1,4}=-a^2+b^2+c^2-d^2$, etc. 
{\em i.e.} the subindices point out the places of the negative signs.
%$\kappa _{-+-+}=-a^2+b^2-c^2+d^2$,  $\kappa _{++--}=a^2+b^2-c^2-d^2$, %$\kappa _{-++-}=-a^2+b^2+c^2-d^2$, etc, 
If $\delta':=(\phi\w\phi)^{-1}\delta\phi$ then it  is given by $\delta'=$
\[
\left(
\begin{array}{ccc}
  \frac{ \alpha \kappa _{1,3} +2\beta ( ac - bd)+\gamma \kappa _{1,4}}{2}&{}^{ -\alpha (ab+cd)+\beta (bc+ad) +\gamma (-ab+cd)}& {}^0\\
{}^0&\frac{\alpha \kappa -2\beta (ac+bd)+c\kappa _{3,4}}{2}&- \frac{ \alpha\kappa _{1,3}
+2\beta (ac-bd)+\gamma\kappa _{1,4}}{2}\\
 \frac{-\alpha \kappa +2\beta (ac+bd)-\gamma\kappa _{3,4}}{2}& {}^0&{}_{\alpha (ab+cd)-\beta(ad+bc) + \gamma (ab-cd)}
\end{array}\right)
\]
%where we use the ordered basis $\{u,v,w\}$ and $\{u\w v,v\w w,w\w u\}$ of $\g$ and %$\Lambda^2\g$ respectively.
If $a=d$ and $c=-b$, with $ad-bc=a^2+b^2=1$, we get
\[
\delta'=
\left(
\begin{array}{ccc}
 -2\beta ab-\gamma (a^2-b^2)&\beta (a^2-b^2) -2 \gamma ab& 0\\
0&\alpha & 2\beta ab+\gamma (a^2-b^2)\\
-\alpha &0&-\beta (a^2-b^2) +2\gamma ab 
\end{array}\right)
\]
Because $a^2+b^2=1$, there exists $\theta\in\R$ such that $a=\cos(\theta)$ and $b=\sin(\theta)$. It
follows that
\[
\begin{array}{rcc}
\beta '&=&\beta (a^2-b^2) -2 \gamma ab=\beta \cos(2\theta)-\gamma \sin(2\theta)\\
\gamma '&=&2\beta ab+\gamma ( a^2-b^2)=\beta \sin(2\theta)+\gamma \cos(2\theta)
\end{array}
\]
Namely, the pair $(\beta ,\gamma)$ transform as a rotation, so we can change it, for example,  into $(\sqrt{\beta ^2+\gamma ^2},0)$. In other words, we can assume that $\gamma =0$ and that $\beta \geq 0$.

\

Now if 
$\delta =\left(
\begin{array}{ccc}
0 & \beta &0\\
0  & \alpha &0\\
-\alpha  &0&-\beta
\end{array}
\right)$ with $\beta\geq 0$, we can take an automorphism with $d=a$, $b=c$
and $ad-bc=a^2-b^2=1$. Such $a$, $b$ may be writen in the form
$a=\cosh \theta$, $b=\sinh\theta$ for a $\theta\in\R$. Under an automorphism like this, $\delta$
changes into
\[
\delta'=
\left(
\begin{array}{ccc}
0&\beta (a^2+b^2)-2\alpha ab&0\\
0&-2\beta ab+\alpha (a^2+b^2)&0\\
2\beta ab-\alpha (a^2+b^2)&0&-\beta(a^2+b^2)+2\alpha ab\\
 \end{array}\right)
\]
This says, for the coefficients of $\delta '$, that $\gamma '=0$, 
\[
\begin{array}{rcc}
\beta '&=&\beta (a^2+b^2)-2\alpha ab=\beta \cosh(2\theta)-\alpha \sinh(2\theta)\\
\alpha '&=&-2\beta ab+\alpha (a^2+b^2)=-\beta\sinh(2\theta)+\alpha \cosh(2\theta)
\end{array}
\]
There are three possibilities:
\begin{enumerate}
\item $\beta ^2-\alpha ^2>0$. In this case, we can choose $\theta$ such that $\alpha '=0$.
\item $\beta ^2-\alpha ^2<0$. In this case, we can choose $\theta$ such that $\beta '=0$.
\item $\beta ^2-\alpha ^2=0$. In this case $\alpha =\pm \beta $. But the automorphism
with $a=d=0$, $b=1=-c$ changes $\alpha '=\alpha $ and $\beta '=-\beta $. So, 
we can assume $\beta =\alpha $.
\end{enumerate}

%\subsubsection{Case $\alpha =0$, $\beta\neq 0$}

\noindent {\em Case $\alpha =0$, $\beta\neq 0$}.
In this case, under a general automorphism, $\delta$ changes into
\[
\delta'=\left(
\begin{array}{ccc}
    \beta (ac-bd)& \beta (ad+bc)&0\\
0     &-\beta (ac+bd) & -\beta (ac-bd)\\
\beta (ac+bd)&0& -\beta (ad+bc)
\end{array}\right)
\]
An automorphism preserving $\alpha  '=0=\gamma '$ must satisfy
$ac=0=bd$. 
If $a\neq 0$ then $c=0$; but  $ad-bc=1$, so $d\neq 0$,
$b=0$ and $d=a^{-1}$; in this case $\delta'=\delta$.

If $a=0$, then $b\neq 0$ since $ad-bc=1$; so $d=0$ and $c=-b^{-1}$. 
In this case $\delta'=-\delta$. We conclude that $\beta$ may be chosen up to a sign,
so a list of representatives of the isomorphism classes is given by
\[\left\{
\delta _\beta=\left(
\begin{array}{ccc}
    0& \beta &0\\
0&0&0\\
0&0&-\beta
\end{array}\right): \beta >0\right\}
\]
For each of these, the automorphism group is
$
\left\{\phi_{S}: S=\left(
\begin{array}{ccc}
   a&0\\
0&a^{-1}\end{array}\right); 0\neq a\in\R\right\}
$.
Notice that $\phi_{S}=\phi_{-S}$, so this group is connected, and,
it is isomorphic to $(\R ,+)$.

\

\noindent {\em Case $\alpha\neq 0$, $\beta =0$:}
Under an automorphism $\phi_S$, $\delta=\left(
\begin{array}{ccc}
    0& 0&0\\
0&\alpha &0\\
-\alpha &0&0
\end{array}\right)\mapsto
\delta'=$
\[=
\frac{\alpha}{2}\left(
\begin{array}{ccc}
-a^2 + b^2 - c^2 + d^2& -2(ab +cd)& 0\\
0& a^2 + b^2 + c^2 + d^2& a^2 - b^2 + c^2 - d^2)\\
-(a^2 + b^2 - c^2 + d^2)& 0& 2 (ab+cd)
\end{array}\right)
\]
If $\beta '=\gamma '=0$ then 
\begin{eqnarray}-a^2 + b^2 - c^2 + d^2=0\\
ab +cd=0\label{eq2}\\
ad -bc=1 
\end{eqnarray}
If $c=0$, then
\[-a^2 + b^2 + d^2=0,\quad
a b  =0,\quad
a d =1\]
The last two equations imply $a\neq 0$, $b=0$, $d=a^{-1}$ and the first one gives
$a^4=1$ then $a=\pm 1$. In both cases, the automorphism acts trivially.

\noindent If $c\neq 0$ we solve $d=-ab/c$ from equation (\ref{eq2}) and the other equations transform into
\[
\frac{(a^2+c^2)(c^2-b^2)}{c^2}=0,
\quad
-b\frac{a^2+c^2}{c}=1
\]
Since $c\neq 0$, the first equation is equivalent to $c^2=b^2$ then $b=\pm c$. But if
$b=c$,
the last equation gives $-(a^2+c^2)=1$, which is absurd, so $b=-c$, hence  $a^2+c^2=1$ and $d=a$.
If this is the case, $\delta'=\delta$. The set of isomorphism classes of this type is the 1-parameter family of representatives given by
\[\left\{\delta _\alpha=
\alpha\left(
\begin{array}{ccc}
0&0&0\\
0&1&0\\
-1&0&0
\end{array}\right): \alpha\in\R,\ \alpha\neq 0\right\}
\]
For each of theses classes, the automorphisms group consists of
\[
\left\{\phi_{S}: S=\left(
\begin{array}{ccc}
   a&-c\\
c&a\end{array}\right),
 a^2+c^2=1\right\}
=
\left\{\left(
\begin{array}{ccc}
   a^2-c^2&-2ac&0\\
2ac&a^2-c^2 & 0\\
0 & 0 & 1
\end{array}\right):
a^2+c^2=1\right\}
\]

%\subsubsection{Case $\alpha =\beta$}

\noindent {\em Case $\alpha =\beta$:}
If  $\delta=\left(
\begin{array}{ccc}
    0&\alpha &0\\
0&\alpha &0\\
-\alpha &0&-\alpha
\end{array}\right)
$
then under an isomorphism $\Phi _S$ we get
\[\delta' =
\alpha
\left(\begin{array}{ccc}
\frac{-a^2 + b^2 + 2      ac - c^2 - 2 bd + d^2}{2}&-2(a-c)(b-d)& 0\\
0&\frac{ a^2 + b^2 - 2      ac + c^2 - 2 bd+d^2}{2}&
\frac{a^2 - b^2 - 2      ac + c^2 +2 bd  -d^2}{2}\\
\frac{- a^2 - b^2 + 2      ac - c^2 + 2 bd-d^2 }{2}& 0& 2 (a -c)(b - d)\\
\end{array}
\right)
\]
%Namely,
%\[\begin{array}{rcl}
%c'&=&\frac{a}{2}(-A^2 + B^2 + 2 A C - C^2 - 2 B D + D^2)\\
%a'&=&\frac{a}{2}(A^2 + B^2 - 2 A C + C^2 - 2 B D + D^2)\\
%b'&=&-a (A - C) (B - D)
%\end{array}
%\]
If one wants to preserve $\alpha '=\beta '$, $\gamma '=0$ then we have two possibilities:
$\alpha =0$, or
\begin{eqnarray*}
-a^2 +b^2 + 2 ac - c^2 - 2 bd + d^2&=&0,\\
a^2 + b^2 - 2 ac + c^2 - 2 bd + d^2&=&-2  (a-c) (b-d)
\end{eqnarray*}
which %These equations 
are equivalent to
$(b-d)^2=-(a-c)(b-d)$ and $(a-c)^2=(c-a)(b-d)$, 
then $c-a=b-d$, so, setting $d=a-c+b$ we get 
$\delta'=\alpha (a-c)^2
\left(
\begin{array}{ccc}
0&1&0\\
0&1&0\\
-1&0&-1
\end{array}
\right)
$.
The condition on $a,b,c,d$ is $ad-bc=1$, then
\[
\begin{array}{ccccc}
1 &=& ad-bc &=& a(a-c+b)-bc\\
&=& a(a+b)-c(a+b) &=&  (a+b)(a-c)
\end{array}
\]
So, there is no restriction on the value of $a-c$, except being different from zero. 
Inside its isomorphism class, $\alpha $ is determined up to positive scalar and 
it is enough to take $\alpha =\pm 1$.
Hence, there are only three isomorphism classes of triangular Lie bialgebras on $\sl(2)$:
\[
\delta _0=\left(\begin{array}{ccc}
0&0&0\\
0&0&0\\
0&0&0
\end{array}
\right);\
\delta _1=\left(\begin{array}{ccc}
0&1&0\\
0&1&0\\
-1&0&-1
\end{array}
\right);\
\delta _{-1}=
\left(\begin{array}{ccc}
0&-1&0\\
0&-1&0\\
1&0&1
\end{array}
\right)
\]
\begin{proposition}
The set of isomorphism classes of Lie bialgebras on $\sl(2,\R)$ are given by the following representatives
\[
\hbox{ Factorizable:\ }
\delta _{\beta}=\left(
\begin{array}{ccc}
    0& \beta  &0\\
0&0&0\\
0&0&-\beta 
\end{array}\right): \beta >0;
\quad
\delta _{\alpha}=\left(
\begin{array}{ccc}
0&0&0\\
0&\alpha &0\\
-\alpha &0&0
\end{array}\right): \alpha \in\R,\ \alpha \neq 0.
\]
\[
\hbox{Triangular:\ }
\delta =0;\quad 
\delta _1=\left(\begin{array}{ccc}
0&1&0\\
0&1&0\\
-1&0&-1
\end{array}
\right)
;\quad 
\delta _{-1}=\left(\begin{array}{ccc}
0&-1&0\\
0&-1&0\\
1&0&1
\end{array}
\right).
\]
Besides, the Lie bialgebra corresponding to $\delta _{\beta}$ verifies
$\g\cong\g ^{cop} $. On the other hand,
although $\delta _{-1}=-\delta _{1}$, the corresponding Lie bialgebras are non isomorhphic.

\end{proposition}

\begin{remark}
For any $r=\alpha u\w v+\beta v\w w+\gamma w\w u\in \Lambda^2\sl(2,\R)$, 
$
[r,r]=2(\alpha ^2-\beta ^2-\gamma ^2)u\w v\w w\in \Lambda^3\sl(2,\R)
=(\Lambda^3\sl(2,\R))^{\sl(2,\R)}
$.
We have the possibilities:
\begin{itemize}
\item $[r,r]=0$ if and only if $\alpha ^2-\beta ^2-\gamma ^2=0$. So, unlike the $\su(2)$ case, there are non trivial triangular structures, explicitly
given by $\delta _{\pm 1}$ with $\alpha=\beta =\pm 1$.
\item $0\neq [r,r]$, and so $(\sl(2,\R),\delta)$ with $\delta (-)=\ad_{(-)}(r)$ is factorizable.
\end{itemize}
The  factorizable structures for $\sl(2,\R)$, up to a scalar multiple, are well known (see \cite{AJ}).
This is in perfect agreement with the results of this section.
\end{remark}


\begin{thebibliography}{99}



\bibitem[A-J]{AJ} N.  Andruskiewitsch, A.P. Jancsa, {\em On Simple Real Lie Bialgebras.} I.M.R.N. 2004 Nro.3.

\bibitem[G]{G} X. Gomez, {\em Classification of three-dimensional Lie bialgebras.}
J. Math. Phys. 41, No.7, 4939-4956 (2000). 

\bibitem[G-O-V]{GOV}
Gorbatsevich, V.V.; Onishchik, A.L.; Vinberg, E.B, {\em  Structure of Lie groups and Lie algebras, 
Lie groups and Lie algebras III}, % (A.L. Onishchik, E.B. Vinberg, eds.),
 Encyclopedia of Mathematical Sciences 41, Springer-Verlag, Berlin,  1994.

\bibitem[K-S]{K-S} L. Korogodsky, Y. Soibelman, {\em Algebras of functions on quantum groups. I}, Amer. Math. Soc., 1997.

\bibitem[RA-H-R]{RA-H-R}  Rezaei-Aghdam, A.; Hemmati, M.; Rastkar, A.R.
Classification of real three-dimensional Lie bialgebras and their Poisson-Lie groups. 
[J] J. Phys. A, Math. Gen. 38, No. 18, 3981-3994 (2005).

\end{thebibliography}
\end{document}